\documentclass[12pt]{amsart}

\usepackage{amsmath, amsfonts, amssymb, amsthm, mathrsfs}
\usepackage[utf8,utf8x]{inputenc}
\usepackage[pagebackref=true]{hyperref}
\hypersetup{%
  pdftitle = {Accurate arrangements},
  pdfkeywords = {},
  pdfsubject = {},
  pdfauthor = {},
  colorlinks = true,
  linktoc = page,
  citecolor = purple,
  linkcolor = blue 
}

\usepackage{xcolor}

\usepackage{caption}
\usepackage{tikz}
\usetikzlibrary{calc}
\usetikzlibrary{intersections}
\usepackage{enumerate}

\usepackage{mfirstuc}

\newcommand{\CC }{\mathbb{C}}
\newcommand{\RR }{\mathbb{R}}

\newcommand{\ZZ }{\mathbb{Z}}
\newcommand{\KK }{\mathbb{K}}

\newcommand{\Ac }{\mathscr{A}}
\newcommand{\Eg }{\mathcal{E}}
\newcommand{\Bc }{\mathscr{B}}
\newcommand{\Dc }{\mathscr{D}}
\newcommand{\Gg }{\mathcal{G}}

\newcommand{\Vg }{\mathcal{V}}

\newcommand{\Cc }{\mathscr{C}}

\newcommand{\Ic}{\mathcal{I}}

\DeclareMathOperator{\Shi}{Shi}
\DeclareMathOperator{\Cat}{Cat}

\DeclareMathOperator{\Der}{Der}
\DeclareMathOperator{\GL}{GL}

\newcommand\ddxi[1]{\partial/\partial x_{#1}}

\DeclareMathOperator{\rk}{rk}

\DeclareMathOperator{\h}{ht}

\numberwithin{equation}{section}

\theoremstyle{plain}
\newtheorem{lemma}[equation]{Lemma}
\newtheorem{theorem}[equation]{Theorem}

\newtheorem{corollary}[equation]{Corollary}
\newtheorem{proposition}[equation]{Proposition}
\theoremstyle{definition}
\newtheorem{definition}[equation]{Definition}
\newtheorem{remark}[equation]{Remark}
\newtheorem{remarks}[equation]{Remarks}
\newtheorem{example}[equation]{Example}

\title[Accurate Arrangements]
{Accurate Arrangements}

\author[P.~M\"ucksch]{Paul M\"ucksch}
\address{Paul M\"ucksch,
Fakult\"at f\"ur Mathematik, Ruhr-Universit\"at Bochum,
D-44780 Bochum, Germany}
\email{paul.muecksch@rub.de}

\author[G. R\"ohrle]{Gerhard R\"ohrle}
\address{Gerhard R\"ohrle,
Fakult\"at f\"ur Mathematik,
Ruhr-Universit\"at Bochum,
D-44780 Bochum, Germany}
\email{gerhard.roehrle@rub.de}

\begin{document}

\begin{abstract}
    Let $\Ac$ be a Coxeter arrangement of rank $\ell$.
    In 1987 Orlik, Solomon and Terao conjectured that
    for every $1\leq d \leq \ell$, the first $d$ exponents of $\Ac$ -- 
    when listed in increasing order -- 
    are realized as the exponents of a free restriction
    of $\Ac$ to some intersection of reflecting hyperplanes of $\Ac$ of dimension $d$.

    This conjecture does follow from 
    rather extensive case-by-case studies by Orlik and Terao from 1992 and 1993, where they 
    show that all restrictions of 
    Coxeter arrangements are free.
    
    We call a general free arrangement with this natural property involving their free restrictions accurate.
    In this paper we initialize their systematic study.
    
    Our principal result shows that MAT-free arrangements, a notion recently introduced by Cuntz and M\"ucksch, are accurate.
    
    This theorem in turn directly implies this special property
    for all ideal subarrangements of Weyl arrangements.
    In particular, 
    this gives a new, simpler and uniform proof 
    of the aforementioned conjecture of Orlik, Solomon and Terao 
    for Weyl arrangements 
    which is free of any case-by-case considerations. 
    
    Another application of a slightly more general formulation of our main theorem
    shows that extended Catalan arrangements, extended Shi arrangements, and 
    ideal-Shi arrangements share this property as well.
    
    We also study arrangements that satisfy a slightly weaker condition, 
    called  almost accurate arrangements, 
    where we simply disregard the ordering of the exponents involved. 
    This property in turn is implied by many well established concepts of freeness such as supersolvability and divisional freeness. 
\end{abstract}

%%%%%%%%%%%%%%%%%%%%%%%%%%%%%%%%%%%%%%%%%%%%%%%%%%%%%%%%%%%%%%%%%%%%%%%%%%%%%%%%%%%%%%%


\maketitle

%%%%%%%%%%%%%%%%%%%%%%%%%%%%%%%%%%%%%%%%%%%%%%%%%%%%%%%%%%%%%%%%%%%%%%%%%%%%%%%%%%%%%%%

\tableofcontents

%%%%%%%%%%%%%%%%%%%%%%%%%%%%%%%%%%%%%%%%%%%%%%%%%%%%%%%%%%%%%%%%%%%%%%%%%%%%%%%%%%%%%%%

\section{Introduction}

The Coxeter arrangement $\Ac (W)$ of a Coxeter group $W$ is the hyperplane arrangement in the reflection representation of $W$ which consists of the reflecting hyperplanes associated with the reflections in $W$.
Coxeter arrangements are pivotal to the theory of hyperplane arrangements. 
In the late 1970s Arnold \cite{Arn76_CoxeterFree1}, \cite{Arn79_CoxeterFree2} 
and Saito \cite{Sai75_CoxeterFree1}, \cite{Sai80_CoxeterFree2} 
showed independently that every Coxeter arrangement is free and that its 
exponents are given by the exponents of the underlying Coxeter group. 

In  \cite{OrSol83_CoxeterArr}, Orlik and Solomon
computed the characteristic polynomial
of every restriction $\Ac (W)^X$ of each Coxeter arrangement $\Ac (W)$ 
to some intersection $X$ of reflecting hyperplanes of $\Ac(W)$
in long and intricate computations.
They showed that each such factors over $\ZZ$.
Moreover, they 
observed that for each $1 \leq d \leq \ell$, where $\ell$ is the dimension of the reflection representation of $W$, there exists 
such a restriction $\Ac (W)^X$ such that the roots of the 
characteristic polynomial of $\Ac (W)^X$ are the first $d$ 
exponents of $\Ac (W)$
when ordered by size. 

In view of this fact, owing to 
\cite[Thm.~1.12]{OST1987_CoxeterArrNumber}
(see Theorem \ref{Thm_RestHWeyl}), 
and Terao's Factorization Theorem for free arrangements
\cite{terao:freefactors} 
(cf.~\cite[Thm.\ 4.137]{OrTer92_Arr}),
Orlik, Solomon and Terao conjectured in  \cite[Conj.~1.14]{OST1987_CoxeterArrNumber} 
that an even stronger property might hold. 
Namely that 
the first $d$ exponents of $\Ac (W)$ 
when ordered by size
are realized as the exponents of some free restriction $\Ac (W)^X$ 
of $\Ac(W)$ for every $1 \leq d \leq \ell$
(Theorem \ref{thm:otsconj}), where $\ell$ is as above.

In long and intricate case-by-case studies, 
Orlik and Terao showed in 
\cite[\S 6.4]{OrTer92_Arr} and 
\cite{OrTer92_CoxeterArrRestr} that indeed 
all the restricted arrangements 
$\Ac(W)^X$ are free for a Coxeter group $W$.
In case $W$ is a Weyl group, Douglass \cite[Cor.~6.1]{douglass:adjoint}
gave a uniform proof of this fact by means of an elegant, conceptual Lie theoretic argument.
But the latter does not provide any information on the exponents of the restrictions at hand. 

From the numerical data obtained 
in \cite{OrSol83_CoxeterArr} and 
the results from  
\cite[\S 6.4]{OrTer92_Arr} and 
\cite{OrTer92_CoxeterArrRestr}
along with \cite[Thm.~1.12]{OST1987_CoxeterArrNumber} 
one can readily confirm the aforementioned conjecture \cite[Conj.~1.14]{OST1987_CoxeterArrNumber} simply by checking all instances 
-- the only known proof of this conjecture to date.
As a consequence of our results,	
we give a new proof of this fact 
in the case of Weyl arrangements 
which is uniform and free of any case-by-case considerations.

The goal of this paper is to initialize a systematic study 
of this natural property in general. 

\begin{definition}
    \label{Def_TF}
    An arrangement $\Ac$ is said to be \emph{accurate}
    if $\Ac$ is free with exponents $\exp(\Ac) = (e_1 \leq e_2 \leq \ldots \leq e_\ell)$ 
    and for every $1 \leq d \leq \ell$ there exists an intersection $X$ of hyperplanes  from  
    $\Ac$ of dimension $d$ such that $\Ac^{X}$ is free with $\exp(\Ac^{X}) = (e_1 \leq e_2 \leq \ldots \leq e_{d})$.
\end{definition}

In \cite{ABCHT16_FreeIdealWeyl}, Abe, Barakat, Cuntz, Hoge and Terao proved the
so-called  \emph{Multiple Addition Theorem} (MAT) (Theorem \ref{Thm_MAT})
which is a variation
of the addition part of Terao's seminal Addition-Deletion Theorem \cite{Terao1980_FreeI} (\cite[Thm.~4.51]{OrTer92_Arr}).
Using this theorem, they went on
to uniformly derive the freeness of \emph{ideal subarrangements} of Weyl arrangements (Definition \ref{DEF_ideal} and Theorem \ref{Thm_IdealFree}).
As a special case of this result, they obtained a new uniform proof of the
classical Kostant-Macdonald-Shapiro-Steinberg formula for the exponents of a Weyl group. 

\bigskip

In \cite{CunMue19_MATfree}, Cuntz and M\"ucksch introduced the notion of \emph{MAT-freeness} (Definition \ref{Def_MATfree})
to investigate arrangements whose freeness can be derived using an iterative application of the Multiple Addition Theorem (Theorem \ref{Thm_MAT}), 
similar to the class of ideal arrangements.

Our principal result asserts that MAT-freeness
is sufficient for accuracy from Definition \ref{Def_TF}.

\begin{theorem}
    \label{Thm_MATRest}
    MAT-free arrangements are accurate.
\end{theorem}

Indeed, we prove a more detailed result in Theorem \ref{Thm_MATRest_sec} as follows. 
Firstly, we determine which restrictions $\Ac^{X}$ of an MAT-free arrangement $\Ac$ realize the subsequences of the
exponents of $\Ac$. They are the intersections of hyperplanes contained in one of the blocks of a certain  partition associated to
every MAT-free arrangement, a so called \emph{MAT-partition} (Definition \ref{Def_MATfree}).

Secondly, we delineate  bases of the derivation module of an MAT-free arrangement $\Ac$ which restrict to
bases of the derivation modules of the relevant restricted arrangements $\Ac^{X}$.
For details see Theorem \ref{Thm_MATRest_sec}.

The methods we use to derive  our results are inspired in part by
arguments employed in the proof of Theorem \ref{Thm_MAT} from \cite{ABCHT16_FreeIdealWeyl}.

\bigskip

As ideal subarrangements of Weyl arrangements are MAT-free, by \cite{ABCHT16_FreeIdealWeyl},  
Theorem \ref{Thm_MATRest}
readily yields our next key result.

\begin{theorem}
    \label{Thm_HtRestIdeal} %
    Ideal arrangements are accurate.
\end{theorem}

The blocks of the associated MAT-partition are realized by roots of the same height
in the given ideal. Hence in the case of ideal arrangements the intersections of hyperplanes
realizing the particular exponents are given by antichains
in the root poset consisting of roots of the same height. 
For further details see Section \ref{SSec_WeylIdeal}.

We record the following special case from Theorem \ref{Thm_HtRestIdeal}.

\begin{corollary}
    \label{Coro_HighestRootWeyl}
    Let $\Ac = \Ac(W)$ be an irreducible Weyl arrangement with $\exp(\Ac) = (e_1 \leq e_2 \leq \ldots \leq e_\ell)$
    and let $H \in \Ac$ be the hyperplane corresponding to the highest root of an irreducible root system associated to $W$.
    Then $\Ac^H$ is free with $\exp(\Ac^H) = (e_1 \leq e_2 \leq \ldots \leq e_{\ell-1})$.
\end{corollary}

The highest root of an irreducible root system associated to the Weyl group $W$ is always long (cf.~\cite[Ch.~VI, Prop.~25]{bourbaki:groupes})
and roots of the same length are conjugate under the action of $W$ (cf.~\cite[Ch.~VI, Prop.~11]{bourbaki:groupes}).
Further, as the linear transformation mapping a root system to its dual, interchanging long and short roots, 
fixes the ambient Weyl arrangement (cf.~\cite[Ch.~VI, Prop.~1 ff.]{bourbaki:groupes}), 
we recover the following classical result already mentioned
above  for the case of Weyl arrangements from Corollary \ref{Coro_HighestRootWeyl}.

\begin{theorem}[{\cite[Thm.~1.12]{OST1987_CoxeterArrNumber}}]
    \label{Thm_RestHWeyl}
    Let $\Ac = \Ac(W)$ be an irreducible Coxeter arrangement
    with $\exp(\Ac) = (e_1 \leq e_2 \leq \ldots \leq e_\ell)$ and $H \in \Ac$.
    Then $\Ac^H$ is free with $\exp(\Ac^H) = (e_1 \leq e_2 \leq \ldots \leq e_{\ell-1})$.
\end{theorem}

Because all Coxeter arrangements are MAT-free, owing to \cite{ABCHT16_FreeIdealWeyl} and
\cite{CunMue19_MATfree}, we obtain our third chief result as a consequence of Theorem \ref{Thm_MATRest}.

\begin{theorem}[{\cite[Conj.~1.14]{OST1987_CoxeterArrNumber}}]
	\label{thm:otsconj}
	Coxeter arrangements are accurate.
\end{theorem}

\begin{remarks}	
    \label{rem_MATRest}
    (i)
        In \cite{ABCHT16_FreeIdealWeyl}, MAT-freeness is derived uniformly
		for the Weyl arrangement case. 
		So our methods give a new uniform proof of Theorem \ref{thm:otsconj} for this case
		which is considerably simpler than the original one indicated above, and most importantly it does not 
		require the classification of the Weyl arrangements nor that of their restrictions.

    (ii)
		The converse of Theorem \ref{Thm_MATRest} is false. 
		The rank $3$ supersolvable simplicial arrangement $\Ac(10,1)$ with $10$ hyperplanes (cf.~\cite{Grue09_SimplArr}) is accurate 
		but it is not MAT-free, by \cite[Ex.~3.10]{CunMue19_MATfree}.
	
    (iii)
		In general an accurate  arrangement $\Ac$ may admit a free restriction whose exponents 
		fail to be a subset of the exponents of $\Ac$, e.g.~the lattice of the reflection arrangement of the Coxeter group of type $E_6$ admits several such instances.
		
		It even might be the case that an accurate arrangement admits a restriction  which is not free at all. For instance, there are 
		examples of ideal subarrangements in the Weyl arrangements of type $D_n$ which admit rank 3 restrictions which fail to be free, cf.~\cite[Rem.~3.6]{amendmoellerroehrle18}.
		In particular, accuracy is not hereditary in general. See also Example \ref{Ex_GraphNotAcc} for such an instance in a graphic arrangement.

    (iv) Also accuracy is not preserved under taking localizations, see Examples \ref{Ex_GraphNotAcc} and \ref{ex:local}.
    
    (v)
		A product of arrangements is accurate if and only if each factor is accurate, cf.~\cite[Prop.~2.14, Prop.~4.28]{OrTer92_Arr}.
\end{remarks}

\bigskip

Our methods actually apply in a slightly more general setting.
Instead of considering only MAT-free arrangements, i.e.\
arrangements build up from the empty arrangement using Theorem \ref{Thm_MAT},
we can instead start with a suitable free arrangement (which need not be empty)
and add hyperplanes successively by means of Theorem \ref{Thm_MAT} (Theorem \ref{Thm_MATRestGen}). 

Utilizing this fact together with a recent observation by Abe and Terao from \cite{AbeTer18_MultAddDelRes},
and a special case of a result by the same authors from \cite{AbeTer15_IdealShi}
gives our fourth main theorem, demonstrating that yet another prominent class of arrangements is accurate,
namely the so called \emph{ideal-Shi arrangements} $\Shi^k_\Ic$ and 
the \emph{extended Catalan arrangements} $\Cat^k$ (Definition \ref{Def_IdealShiCatalan}).
For details, see Section \ref{SSec_IdealShiCatalan}.
\begin{theorem}
    \label{Thm_ResShiCatTF}
    Extended Shi arrangements $\Shi^k$, ideal-Shi arrangements $\Shi^k_\Ic$ and extended Catalan arrangements $\Cat^k$ are accurate.
\end{theorem}

\bigskip

In Section \ref{ssect:divfree},
we study arrangements that satisfy a slightly weaker property, so called  almost accurate arrangements, see Definition \ref{Def_aTF},  
where compared to accurate arrangements, 
we simply disregard the ordering of the exponents involved. 
This property in turn is 
easily seen to be implied by many well established concepts of freeness such as supersolvability or divisional freeness, for instance. 

However, in contrast it turns out that in general, accuracy is not implied by supersolvability (so neither by inductive freeness nor by divisional freeness), see Example \ref{Ex_SSbnTF}.

On the other hand there are accurate arrangements which are not divisionally free, see Example \ref{ex:star-but-not-divfree}.
So in particular, accuracy is not implied by divisional freeness in general. 
Consequently, the recent result by Cuntz, R\"ohrle and Schauenburg \cite{CRS17_IdealIF}, asserting that all ideal arrangements
are divisionally free does not imply Theorem \ref{Thm_HtRestIdeal}.

\bigskip
In Section \ref{ssect:graph}, we consider free graphic arrangements under the aspect of 
accuracy.
In Example \ref{Ex_GraphNotAcc}, we present a non-accurate free graphic arrangement which admits an extension by one vertex which gives an accurate graphic arrangement.
This example also shows that accuracy is neither preserved under localizations nor under restrictions.

\bigskip
In view of the original theme of 
Orlik, Solomon and Terao for Coxeter arrangements (Theorem \ref{thm:otsconj}), 
we consider 
complex reflection arrangements and their restrictions in our final Section \ref{ssect:refl}, and 
determine all accurate members among them
(see Theorems \ref{thm:complex} and \ref{thm:restr}).
Along the way we show that accuracy from our motivating case of Coxeter arrangements (Theorem \ref{thm:otsconj}) descends to restrictions of the latter
(Corollary \ref{cor:restr-Coxeter}).

\bigskip

For general information about arrangements, Weyl groups and root systems,  
we refer the reader to \cite{bourbaki:groupes} and 
\cite{OrTer92_Arr}.

%%%%%%%%%%%%%%%%%%%%%%%%%%%%%%%%%%%%%%%%%%%%%%%%%%%%%%%%%%%%%%%%%%%%%%%%%%%%%%%%%%%%%%%
%%%%%%%%%%%%%%%%%%%%%%%%%%%%%%%%%%%%%%%%%%%%%%%%%%%%%%%%%%%%%%%%%%%%%%%%%%%%%%%%%%%%%%%

\section{Preliminaries}
\label{ssect:recoll}

\subsection{Hyperplane arrangements}
\label{ssect:hyper}

Let $\Ac = (\Ac, V)$ be a hyperplane arrangement in $V \cong \KK^\ell$ where $\KK = \RR$ or $\KK = \CC$.
If we want to emphasize the dimension $\ell$ of the ambient vector space
we say that $\Ac$ is an $\ell$-arrangement. 
The \emph{intersection lattice} of $\Ac$ is 
$L(\Ac) = \left\{ \cap_{H \in \Bc} H \mid \Bc \subseteq \Ac \right\}$.

Associated with $X \in L(\Ac)$ we have two canonical arrangements, 
the \emph{localization $\Ac_X$} of $\Ac$ at $X$, given by
\[
\Ac_X := \{H \in \Ac  \mid H \supseteq X \},
\]
and the \emph{restriction $\Ac^X$ of $\Ac $ to $X$}, defined by
\[
\Ac^X := \{H \cap X  \mid H \in \Ac\setminus \Ac_X  \}.
\]

The \emph{rank} of $X \in L(\Ac)$ is defined as $\rk(X) := \ell - \dim(X)$
and the rank of $\Ac$ as $\rk(\Ac) = \rk(\cap_{H \in \Ac}H)$.

Let $S = S(V^*)$ be the symmetric algebra of the dual space $V^*$.
We fix a basis $x_1,\ldots,x_\ell$ for $V^*$ and identify $S$ with the polynomial ring $\KK[x_1,\ldots,x_\ell]$.
The algebra $S$ is equipped with the grading by polynomial degree: $S = \bigoplus_{p\in \ZZ} S_p$,
where $S_p$ is the $\KK$-space of homogeneous polynomials of degree $p$ (along with $0$), where $S_p = \{0\}$ for $p < 0$.

\subsection{Free arrangements}
\label{SSec_freeArr}

A $\KK$-linear map $\theta:S\to S$ which satisfies $\theta(fg) = \theta(f)g + f\theta(g)$ is called a $\KK$-\emph{derivation}.
Let $\Der(S)$ be the $S$-module of $\KK$-derivations of $S$. 
It is a free $S$-module with basis $\ddxi{1},\ldots,\ddxi{\ell}$.
We say that $\theta \in \Der(S)$ is \emph{homogeneous of degree} $p$ provided
$\theta = \sum_{i=1}^\ell f_i \ddxi{i}$ with $f_i \in S_p$ for each $1 \leq i \leq \ell$.
In this case we write $\deg{\theta} = p$.
We obtain a $\ZZ$-grading $\Der(S) = \bigoplus_{p \in \ZZ} \Der(S)_p$ of the $S$-module $\Der(S)$.

\begin{definition}
For $H \in \Ac$ we fix $\alpha_H \in V^*$ with $H = \ker(\alpha_H)$.
The \emph{module of $\Ac$-derivations} is defined by
\begin{equation*}
  D(\Ac) := \{ \theta \in \Der(S) \mid \theta(\alpha_H) \in \alpha_H S \text{ for all } H \in \Ac\}.
\end{equation*}
In particular, if $\Bc \subseteq \Ac$, then $D(\Ac) \subseteq D(\Bc)$.

We say that $\Ac$ is \emph{free} if the module of $\Ac$-derivations is a free $S$-module. 
\end{definition}

Let $X \in L(\Ac)$ be of rank $q$ and let 
$P_X := \langle\alpha_H \mid H \in \Ac_X\rangle_S$ be the prime ideal generated by the defining linear forms of
the localization $\Ac_X$. After a possible base change we may assume that $X = \ker(x_{\ell-q+1})\cap \ldots \cap \ker(x_{\ell})$.
Hence $P_X = \langle x_{\ell-q+1},\ldots,x_\ell \rangle_S$ and then $S^X := S / P_X \cong \KK[x_1,\ldots,x_{\ell-q}]$.
The module $D(\Ac^X)$ is naturally an $S^X$-submodule of $\Der(S^X)$.
For $\theta \in D(\Ac)$ we have $\theta(P_X) \subseteq P_X$
and hence we get a \emph{restriction map}  
\[
D(\Ac) \to D(\Ac^X), \theta \mapsto \theta^X \quad\text{ by }\quad  
\theta^X(f + P_X) = \theta(f) + P_X. 
\]
The derivation $\theta^X \in D(\Ac^X)$ is 
the \emph{restriction of $\theta$ to $X$}, 
cf.~ \cite[\S~4.3]{OrTer92_Arr}.
Moreover, if $\theta$ is homogeneous of degree $\deg{\theta} = d$ and 
$\theta^X \neq 0$ then $\deg(\theta^X) = d$.

If $\Ac$ is a free arrangement we may choose a homogeneous basis $\theta_1, \ldots, \theta_\ell$ of $D(\Ac)$.
Then the degrees of the $\theta_i$ are called the \emph{exponents} of $\Ac$. 
They are uniquely determined by $\Ac$, \cite[Def.\ 4.25]{OrTer92_Arr}. 
In that case we write 
\[
    \exp(\Ac) := (\deg{\theta_1},\ldots,\deg{\theta_\ell})
\]
for the exponents of $\Ac$.

Note that the empty arrangement $\varnothing_\ell$ in $V$ is free with
$D(\varnothing_\ell) = \Der(S)$ so that $\exp(\varnothing_\ell)=(0,\ldots,0)\in\ZZ^\ell$.

If $\theta_1, \ldots, \theta_\ell \in \Der(S)$ then
\[
    M(\theta_1, \ldots, \theta_\ell) := \left(\theta_j(x_i)\right)_{1\leq i,j \leq \ell}
\]
denotes the \emph{coefficient matrix} of the $\theta_i$, i.e.\ the matrix of coefficients with respect to the standard basis
$\ddxi{1},\ldots,\ddxi{\ell}$ of $\Der(S)$.

Let $\{\alpha_H \mid H \in \Ac \}$ be defining linear forms for the
hyperplanes in $\Ac$. Then
\[
    Q(\Ac) = \prod_{H \in \Ac} \alpha_H \in S
\]
is the \emph{defining polynomial} of $\Ac$.
It is uniquely determined by $\Ac$ up to a non-zero constant factor. 

Next we recall Saito's criterion, cf.~\cite[Thm.~4.19]{OrTer92_Arr}.

\begin{theorem}
	\label{Thm_SaitosCriterion}
    For $\theta_1, \ldots, \theta_\ell \in D(\Ac)$, the following are equivalent:
    \begin{enumerate}[(1)]
        \item 
        $\det(M(\theta_1, \ldots, \theta_\ell)) \in  \KK^\times Q(\Ac)$,
        
        \item
        $\theta_1, \ldots, \theta_\ell$ is a basis of $D(\Ac)$.
    \end{enumerate}
\end{theorem}

The corollary to the following result from \cite[\S 4]{OrTer92_Arr}  provides
a convenient degree criterion to determine when a derivation of a deletion  
$\Ac \setminus \{H\}$ of $\Ac$ does belong to the smaller $S$-module $D(\Ac)$.  

\begin{proposition}[{\cite[Prop.~4.41]{OrTer92_Arr}}]\label{Prop_bPoly}
    Let $H = \ker(\alpha) \in \Ac$, $\Ac' = \Ac \setminus \{H\}$ and $\Ac'' = \Ac^H$.
    Then there is a homogeneous polynomial $b$ in $S$ of degree $|\Ac'|-|\Ac''|$ such that
    \[
        \theta(\alpha) \in \alpha S + bS
    \]
    for all $\theta \in D(\Ac')$.
\end{proposition}

\begin{corollary}\label{Coro_degTheta}
	With the notation as in Proposition \ref{Prop_bPoly},
    if $\theta \in D(\Ac')$ with $\deg(\theta) < \deg(b) = |\Ac'|-|\Ac''|$
    then $\theta \in D(\Ac)$.
\end{corollary}

We require a construction from \cite[\S 4]{OrTer92_Arr} that allows us to 
extend $S$-independent members from $D(\Ac)$ to a basis of $D(\Ac)$ in case $\Ac$ is free.
Here and later on we write $(e_1,\ldots,e_\ell)_\leq$ for 
$(e_1 \leq e_2 \leq \ldots \leq e_\ell)$.
 
\begin{proposition}[{\cite[Thm.~4.42]{OrTer92_Arr}}]\label{Prop_ContinueBasisD}
	Let $\Ac$ be a free arrangement with $\exp(\Ac) = (e_1,\ldots,e_\ell)_\leq$, 
	and let $\theta_1,\ldots,\theta_k \in D(\Ac)$ be homogeneous elements such that
	\begin{enumerate}[(1)]
		\item $\deg(\theta_i) = e_i$,
		\item $\theta_i \notin S\theta_1 + \ldots + S\theta_{i-1}$.
	\end{enumerate}
	Then $\theta_1,\ldots,\theta_k$ may be extended to a basis of $D(\Ac)$. 
\end{proposition}

We finally recall the restriction part of Terao's seminal Addition-Deletion Theorem \cite{Terao1980_FreeI}.

\begin{proposition}[{\cite[Cor.~4.47]{OrTer92_Arr}}]
    \label{Prop_AddDelRestr}
    Let $H \in \Ac$ and $\Ac' = \Ac \setminus \{H\}$ such that both $\Ac$ and $\Ac'$ are free. Then $\Ac^H$ is
    free and we have
    \begin{align*}
        \exp(\Ac) & = (e_1,\ldots,e_{\ell-1},e_\ell), \\
        \exp(\Ac')& = (e_1,\ldots,e_{\ell-1},e_\ell-1), \\
        \exp(\Ac^H) & = (e_1,\ldots,e_{\ell-1}).        
    \end{align*}
\end{proposition}

%%%%%%%%%%%%%%%%%%%%%%%%%%%%%%%%%%%%%%%%%%%%%%%%%%%%%%%%%%%%%%%%%%%%%%%%%%%%%%%%%%%%%%%

\subsection{Topological incarnation of the module of $\Ac$-derivations}
\label{SSec_tangent}

The \emph{complement} of $\Ac =(\Ac,V)$ is denoted by
\[ 
M(\Ac) := V \setminus \bigcup_{H \in \Ac} H.
\]
For $z \in V$ let $T_{V,z}$ be the \emph{tangent space} of $V$ at $z$;
it is a $\KK$-vector space with basis $\ddxi{1},\ldots,\ddxi{\ell}$.
The $\KK$-linear \emph{evaluation map} 
\[
\rho_z:D(\Ac) \to T_{V,z}
\]
is defined as follows, cf.~\cite[\S 5.1]{OrTer92_Arr}.
For $\theta = \sum_{i=1}^\ell f_i \ddxi{i} \in D(\Ac)$, let 
\[ \rho_z(\theta) := \sum_{i=1}^\ell f_i(z) \ddxi{i}. \]

\begin{proposition}[{\cite[Prop.~5.17]{OrTer92_Arr}}]
	\label{Prop_TangentSpace}
For $z \in M(\Ac^X)$, we have $\rho_z(D(\Ac)) = T_{X,z}$ 
which is isomorphic to $X$ as a $\KK$-vector space.
\end{proposition}

\begin{corollary}
	\label{Coro_TangentSpace}
    If $\Ac$ is free, $\theta_1,\ldots,\theta_\ell$ is a basis of $D(\Ac)$,
    and $z \in M(\Ac)$, then 
    \[
        \rho_z(D(\Ac)) = T_{V,z} = \langle \rho_z(\theta_1),\ldots,\rho_z(\theta_\ell) \rangle_\KK,
    \]
    i.e.\ $\rho_z(\theta_1),\ldots,\rho_z(\theta_\ell)$ is a $\KK$-vector space
    basis of $T_{V,z}$.
\end{corollary}
\begin{proof} 
    Note that for $\theta = \sum_{i=1}^\ell f_i \theta_i$
    we have $\rho_z(\theta) =$ $\sum_{i=1}^\ell f_i(z) \rho_z(\theta_i)$.
    By Proposition \ref{Prop_TangentSpace}, the $\KK$-linear map $\rho_z$ is onto $T_{V,z}$
    and since $\theta_1,\ldots,\theta_\ell$ is a basis of $D(\Ac)$, for
    every $v \in T_{V,z}$ there are $f_1,\ldots,f_\ell \in S$
    such that $v = \rho_z(\sum_{i=1}^\ell f_i \theta_i) = \sum_{i=1}^\ell f_i(z) \rho_z(\theta_i)$.
    Consequently, $\rho_z(\theta_1),\ldots,\rho_z(\theta_\ell)$
    is a basis of the vector space $T_{V,z}$.
\end{proof}

The next fact provides a central part in
the proof of Theorem \ref{Thm_MATRest}.

\begin{proposition}\label{Prop_FreeSubarrIneqRk}
    Let $\Bc \subseteq \Ac$ be free arrangements. Assume that
    $\theta_1,\ldots,\theta_\ell$ is a basis of $D(\Bc)$ such that
    $\theta_i \in D(\Ac)$ for $1\leq i \leq \ell-q$.
    Suppose $X \in L(\Ac)$ such that $\Ac_X \cap \Bc = \varnothing$.
    Then $\rk(X) \leq q$.
\end{proposition}
\begin{proof}
    Since $\Ac_X \cap \Bc = \varnothing$, we see that $M(\Ac^X) \cap M(\Bc) \ne \varnothing$. 
    Let  $z \in M(\Ac^X) \cap M(\Bc)$.
    Then thanks to Proposition \ref{Prop_TangentSpace}, we have $\rho_z(D(\Ac)) = T_{X,z}$.
    But also the tangent vectors $\rho_z(\theta_1),\ldots,\rho_z(\theta_\ell)$ form
    a vector space basis of $T_{V,z}$,     by Corollary \ref{Coro_TangentSpace}. 
    By our assumption the $\theta_i$ belong to $D(\Ac)$ for $1\leq i \leq \ell-q$.
    So $\rho_z(\theta_1),\ldots,\rho_z(\theta_{\ell-q}) \in T_{X,z}$ and they are linearly independent.
    Hence $\dim(X) = \dim(T_{X,z}) \geq \ell-q$, i.e., $\rk(X) \leq q$.
\end{proof}

%%%%%%%%%%%%%%%%%%%%%%%%%%%%%%%%%%%%%%%%%%%%%%%%%%%%%%%%%%%%%%%%%%%%%%%%%%%%%%%%%%%%%%%

\subsection{MAT-free arrangements}
\label{SSec_MAT}

We begin by recalling the core result from \cite{ABCHT16_FreeIdealWeyl}, the so-called Multiple Addition Theorem (MAT).

\begin{theorem}[{\cite[Thm.~3.1]{ABCHT16_FreeIdealWeyl}}]
    \label{Thm_MAT}
    Let $\Ac' = (\Ac', V)$ be a free arrangement with
    $\exp(\Ac')=(e_1,\ldots,e_\ell)_\le$
    and let $1 \le p \le \ell$ be the multiplicity of the highest exponent, i.e.
    \[ e_{\ell-p} < e_{\ell-p+1} =\cdots=e_\ell=:e. \]
    Let $H_1,\ldots,H_q$ be hyperplanes in $V$ with
    $H_i \not \in \Ac'$ for $i=1,\ldots,q$. Define
    \[ \Ac''_j:=(\Ac'\cup \{H_j\})^{H_j}=\{H\cap H_{j} \mid H\in \Ac'\}, \quad \text{ for }j=1,\ldots,q. \]
    Assume that the following conditions are satisfied:
    \begin{itemize}
        \item[(1)]
            $X:=H_1 \cap \cdots \cap H_q$ is $q$-codimensional.
        \item[(2)]
            $X \not \subseteq \bigcup_{H \in \Ac'} H$.
        \item[(3)]
            $|\Ac'|-|\Ac''_j|=e$ for $1 \le j \le q$.
    \end{itemize}
    Then $q \leq p$ and $\Ac:=\Ac' \cup \{H_1,\ldots,H_q\}$ is free with
    $$\exp(\Ac)=(e_1,\ldots,e_{\ell-q},e+1,\ldots,e+1)_\le.$$
\end{theorem}

We often consider the addition of several hyperplanes using
Theorem \ref{Thm_MAT}. This motivates the next terminology.

\begin{definition}
    Let $\Ac'$ and $\{H_1,\ldots,H_q\}$ be as in Theorem \ref{Thm_MAT} such that
    conditions (1)--(3) are satisfied. Then the addition 
    of $\{H_1,\ldots,H_q\}$ to $\Ac'$ resulting in $\Ac = \Ac' \cup \{H_1,\ldots,H_q\}$
    is called an \emph{MAT-step}.
\end{definition}

The following lemma about a certain basis of the derivation module across an MAT-step 
follows directly from the last lines of the proof of Theorem \ref{Thm_MAT} (cf.~\cite[Thm.~3.1]{ABCHT16_FreeIdealWeyl}).

\begin{lemma}
    \label{Lem_BasisMATStep}
    Let $\Ac'$ be free with $\exp(\Ac') = (e_1,\ldots,e_{\ell-p},e,\ldots,e)_\le$ and 
    let $\Ac = \Ac' \dot{\cup} \{H_1=\ker(\alpha_1),\ldots,H_p=\ker(\alpha_p)\}$ be an MAT-step.
    
    Then there is a basis $\theta_1,\ldots,\theta_{\ell-p}$, $\eta_1,\ldots,\eta_p$ 
    of $D(\Ac')$ with $\deg(\theta_i) = e_i$ for $1\leq i \leq \ell-p$ and
    $\deg(\eta_j) = e$ for $1\leq j \leq p$,
    such that for any subset $N \subseteq \{1,\ldots,p\}$ and its complement $\overline{N} = \{1,\ldots,p\} \setminus N$, 
    the derivations
    \[
        \theta_1,\ldots,\theta_{\ell-p},\{\alpha_j\eta_j \}_{j \in \overline{N}},\{\eta_i\}_{i \in N}
    \]
    form a basis of $D(\Ac'\cup\{H_j \mid j \in \overline{N}\})$.
    In particular, $\Ac'\cup\{H_j \mid j \in \overline{N}\}$ is free. 
\end{lemma}

As a consequence of Lemma \ref{Lem_BasisMATStep} and 
a  successive application of Proposition \ref{Prop_AddDelRestr} we obtain the following.

\begin{corollary}
    \label{Coro_MAT-stepRestriction}
    Let $\Ac'$ be free, $\Ac = \Ac' \dot{\cup} \{H_1,\ldots,H_p\}$ an MAT-step
    and $\Cc \subseteq \{H_1,\ldots,H_p\}$.
    Suppose that $\exp(\Ac) = (e_1,\ldots,e_\ell)_\leq$ and let
    $X := \cap_{H\in\Cc}H$.
    Then $\Ac^X$ is free with $\exp(\Ac^X) = (e_1,\ldots,e_{\ell-|\Cc|})_\leq$.
\end{corollary}

An iterative application of Theorem \ref{Thm_MAT} motivates the following natural concept.

\begin{definition}%
    [{\cite[Def.~3.2, Lem.~3.8]{CunMue19_MATfree}}]
    \label{Def_MATfree}%

    An arrangement $\Ac$ is called \emph{MAT-free} if there exists an ordered partition
    \[
        \pi = (\pi_1|\cdots|\pi_n)
    \] 
    of $\Ac$ such that the following hold. 
    Set $\Ac_0 := \varnothing_\ell$ and
    \[
        \Ac_k := \bigcup_{i=1}^k \pi_i \quad\text{ for } 1 \leq k \leq n.
    \]
    Then for every $0 \leq k \leq n-1$ suppose that
    \begin{itemize}
    \item[(1)] $\rk(\pi_{k+1}) = \vert \pi_{k+1} \vert$,
    \item[(2)] $\cap_{H \in \pi_{k+1}} H \nsubseteq \bigcup_{H' \in \Ac_k}H'$,
    \item[(3)] $\vert \Ac_k \vert - \vert (\Ac_k \cup \{H\})^H \vert = k$ for each $H \in \pi_{k+1}$,
    \end{itemize}
    i.e.\ $\Ac_{k+1} = \Ac_{k}\cup\pi_{k+1}$ is an MAT-step.
    
    An ordered partition $\pi$ with these properties is called an \emph{MAT-partition} for $\Ac$.
\end{definition}

Note that in \cite{CunMue19_MATfree} MAT-freeness was defined differently.
However, for our purpose its characterization in \cite[Lem.~3.8]{CunMue19_MATfree} is sufficient.
Hence we take the latter here for our definition.

\begin{remark}
    \label{rem:MAT-free}
    Suppose that $\Ac$ is MAT-free with MAT-partition
    $\pi = (\pi_1|\cdots|\pi_n)$. Then we have:
    \begin{enumerate}[(a)]
        \item 
        for each $1 \leq k \le n$, $\Ac_k$ is MAT-free with MAT-partition $(\pi_1|\cdots|\pi_{k})$, 
        
        \item
        $\Ac$ is free and the exponents 
        $\exp(\Ac) = (e_1,\ldots,e_\ell)_\le$ of $\Ac$ are given by the block sizes
        of the dual partition of $\pi$: 
        \[ 
            e_i := |\{k \mid |\pi_k|\geq \ell-i+1 \}|, 
        \]
        
        \item 
        $|\pi_1| > |\pi_2| \geq \cdots \geq |\pi_n|$.
    \end{enumerate}
\end{remark}
\begin{proof}
    Statement (a) is clear by Definition \ref{Def_MATfree}.
    Statements (b) and (c) follow readily from Theorem \ref{Thm_MAT} and a simple induction.
\end{proof}

%%%%%%%%%%%%%%%%%%%%%%%%%%%%%%%%%%%%%%%%%%%%%%%%%%%%%%%%%%%%%%%%%%%%%%%%%%%%%%%%%%%%%%%
%%%%%%%%%%%%%%%%%%%%%%%%%%%%%%%%%%%%%%%%%%%%%%%%%%%%%%%%%%%%%%%%%%%%%%%%%%%%%%%%%%%%%%%

\section{Proof of Theorem \ref{Thm_MATRest}}
\label{Sec_ProofMain}

We begin with a closer investigation of MAT-free arrangements.
The idea of our proof is to control a basis of the module of $\Ac$-derivations through MAT-steps 
and to apply Proposition \ref{Prop_FreeSubarrIneqRk}.

This in turn leads to the identification of a restricted subarrangement 
with a restriction of the whole arrangement.
Since the former is free with the right exponents, by Corollary \ref{Coro_MAT-stepRestriction},
the freeness of the restriction of the entire arrangement follows.

\subsection{Basis derivations across MAT-steps}
\label{SSec_ProofMainTechPrelim}

First, we recall a technical result which also provided a key step in
the proof of Theorem \ref{Thm_MAT} (cf.~\cite[p.~1343]{ABCHT16_FreeIdealWeyl}).

\begin{lemma}
    \label{Lem_MATstepMatrixC}
    Let $\Ac'$ be free with 
    $\exp(\Ac') = (e_1,\ldots,e_{\ell-p}$, $e,\ldots,$ $e)_\leq$, $e_{\ell-p} < e$, 
    and $\Ac = \Ac' \dot{\cup} \{H_1=\ker(\alpha_1),\ldots,H_r=\ker(\alpha_r)\}$ be an MAT-step. 
    Let $\theta_1,\ldots,\theta_{\ell-p}$, 
    $\varphi_1,\ldots,\varphi_p$ be a homogeneous basis
    of $D(\Ac')$ with $\deg(\theta_i) = e_i$ for $1\leq i \leq \ell-p$ and
    $\deg(\varphi_i) =e$ for $1\leq i \leq p$.
    
    Then there are a matrix $C = (c_{ij}) \in \KK^{p\times r}$ of rank $r$ 
    and homogeneous polynomials $b_j$ of degree $e$, for $1\leq j \leq r$, such that
    \begin{align*}
        \varphi_i(\alpha_j) \equiv c_{ij}b_j \mod \alpha_j S \quad\text{ for } 1\leq i \leq p, 1\leq j \leq r.
    \end{align*}
\end{lemma}

With this lemma we are able to maintain a certain part of a
given basis of the derivation modules while performing MAT-steps.

\begin{lemma}\label{Lem_MATstepPartOfBasis}
    Suppose that $\Ac'$ is free with
    $\exp(\Ac') = (e_1,\ldots,e_{\ell-p}$, $e,\ldots,e)_\leq$, $e_{\ell-p} < e$, 
    and that $\Ac = \Ac' \dot{\cup} \{H_1,\ldots,H_r\}$ is an MAT-step.
    Let $\theta_1,\ldots,\theta_{\ell-p}$, 
    $\varphi_1,\ldots,\varphi_p$ be a homogeneous basis
    of $D(\Ac')$ with $\deg(\theta_i) = e_i$ for $1\leq i \leq \ell-p$ and
    $\deg(\varphi_i) = e$ for $1\leq i \leq p$.
    
    Then for each $r\leq q \leq p$ there are a subset $N \subseteq \{1,\ldots,p\}$ with $|N|=q$ 
    and $\psi_j \in D(\Ac)$, for $j \in \overline{N} = \{1,\ldots,p\}\setminus N$,
    such that $$\theta_1,\ldots,\theta_{\ell-p},\{\psi_j\}_{j \in \overline{N}},\{\varphi_i\}_{i \in N}$$
    form a basis of $D(\Ac')$.
\end{lemma}
\begin{proof}
    By Lemma \ref{Lem_MATstepMatrixC} there are a matrix $C = (c_{ij}) \in \KK^{p\times r}$ of rank $r$ 
    and homogeneous polynomials $b_j$ of degree $e$, for $1\leq j \leq q$, such that
    \begin{align*}
        \varphi_i(\alpha_j) \equiv c_{ij}b_j \mod \alpha_j S \quad\text{ for } 1\leq i \leq p, 1\leq j \leq r.
    \end{align*}
    So for $r\leq q \leq p$ there is a subset $N = \{i_1,\ldots,i_q\} \subseteq \{1,\ldots,p\}$ such that
    the $q$ rows $(c_{i1},\ldots,c_{ir})$, for  $i \in N$ of the matrix $C$ contain a basis of the whole row space of $C$.
    Now, by elementary row operations there is a regular matrix $T\in \GL_p(\KK)$ 
    such that the transformation of the basis given by $T$ fixes the derivations $\varphi_i$ for $i \in N$:
    \begin{align*}
        &(\varphi_1,\ldots,\varphi_p) \cdot T  \\
        = \, &(\psi_1,\ldots,\psi_{i_1-1},\varphi_{i_1},\psi_{i_1+1},\ldots,\psi_{i_q-1},\varphi_{i_q},\psi_{i_q+1},\ldots,\psi_p),
    \end{align*}
    and 
    \[
        \psi_i(\alpha_j) \equiv 0 \mod \alpha_j S \quad\text{ for all } i \in \overline{N}, 1\leq j \leq r,
    \]
    i.e.\ $\psi_i \in D(\Ac)$ for all $i \in \overline{N}$.
    We further have
    \begin{align*}
        &\det\left(M(\theta_1,\ldots,\theta_{\ell-p},\{\psi_j\}_{j \in \overline{N}},\{\varphi_i\}_{i \in N})\right) \\
        = \, &\det(T)\det\left(M(\theta_1,\ldots,\theta_{\ell-p},\varphi_1,\ldots,\varphi_p)\right).
    \end{align*}
    Since $\theta_1,\ldots,\theta_{\ell-p},\varphi_1,\ldots,\varphi_p$ form a basis 
    of $D(\Ac')$ and $\det(T) \in \KK^\times$, 
    by Theorem \ref{Thm_SaitosCriterion}, 
    $\theta_1,\ldots,\theta_{\ell-p},\{\psi_j\}_{j \in \overline{N}},\{\varphi_i\}_{i \in N}$ 
    is again a basis of $D(\Ac')$, as claimed.
\end{proof}

The following technical proposition provides a  crucial step in the proof of Theorem \ref{Thm_MATRest}.
\begin{proposition}
    \label{Prop_TwoMATsteps}
    Let $\Bc'$ be free with 
    $\exp(\Bc') = (e_1,\ldots,e_{\ell-p},e,\ldots,e)_\leq$. Suppose that 
    \begin{enumerate}[(a)] 
        \item
        $\Ac' = \Bc' \dot{\cup} \{H_1,\ldots,H_p\}$ is an MAT-step such that
        $$\exp(\Ac')= (e_1, \ldots,e_{\ell-p},e+1,\ldots,e+1)_\leq,$$
        
        \item
        $\Ac = \Ac' \dot{\cup} \{K_1$, $\ldots,K_r\}$ is an MAT-step
        such that 
        $$\exp(\Ac) = (e_1,\ldots,e_{\ell-p},e+1,\ldots,e+1,e+2,\ldots,e+2)_\leq.$$
    \end{enumerate}
    Then for each $r\leq q \leq p$ there are a $\Cc \subseteq \{H_1,\ldots,H_p\}$ with $|\Cc|=q$ and
    a basis $\theta_1,\ldots,\theta_\ell$ of $D(\Bc)$, where $\Bc = \Ac'\setminus \Cc$,
    such that $\theta_i \in D(\Ac)$ for $1\leq i \leq \ell-q$, 
    $\deg(\theta_i) = e_i$ for $1\leq i \leq \ell-p$,
    and $\deg(\theta_j) = e+1$ for $\ell-p+1 \leq j \leq \ell-q$.
\end{proposition}
\begin{proof}
    Let $H_1 = \ker(\alpha_1), \ldots, H_p =  \ker(\alpha_p)$. 
    By Lemma \ref{Lem_BasisMATStep} applied to the first MAT-step $\Ac' = \Bc' \dot{\cup} \{H_1,\ldots,H_p\}$, 
     there exists a homogeneous basis 
    $\theta_1,\ldots,\theta_{\ell-p},\eta_1,\ldots,\eta_p$ of $D(\Bc')$
    with $\deg(\theta_i) = e_i$, $1\leq i \leq \ell-p$ and 
    $\deg(\eta_j) = e$, $1\leq j \leq p$ such that
    $\theta_1,\ldots,\theta_{\ell-p},\alpha_1\eta_1,\ldots,\alpha_p\eta_p$
    form a homogeneous basis of $D(\Ac')$.
    
    Note that, since $\deg \theta_i < e+1 = |\Ac'| - |(\Ac'\cup \{K_j\})^{K_j}|$ 
    (by condition (3) of Theorem \ref{Thm_MAT} within the second MAT-step),
    we have $\theta_i \in D(\Ac)$ for $1\leq i \leq \ell-p$, by Corollary \ref{Coro_degTheta}.
    
    Thanks to Lemma \ref{Lem_MATstepPartOfBasis}, applied to the second MAT-step,  
    for each $r\leq q \leq p$, there are a subset
    $N\subseteq \{1,\ldots,p\}$ with $|N|=q$ and derivations  $\psi_j \in D(\Ac)$, 
    for $j \in \overline{N} = \{1,\ldots,p\}\setminus N$ such that
    $\theta_1,\ldots,\theta_{\ell-p}$, $\{\psi_j\}_{j \in \overline{N}},\{\alpha_i\eta_i\}_{i \in N}$
    form a basis of $D(\Ac')$. 
    
    Next we define $$\Cc := \{H_i \mid i \in N\}\quad \text{ and }\quad \Bc := \Ac'\setminus \Cc,$$
    and we set
    \begin{align*}
        \{\theta_{\ell-p+1}, \ldots, \theta_{\ell-q} \} &:= \{\psi_j \mid j \in \overline{N} \}, \text{ and} \\
        \{\theta_{\ell-q+1}, \ldots, \theta_\ell \} &:= \{\eta_i \mid i \in N\}.
    \end{align*}
    
    According to Lemma \ref{Lem_BasisMATStep} applied to the first MAT-step, we observe that 
    $\theta_1,\ldots,\theta_{\ell-p},\{\alpha_j\eta_j\}_{j \in \overline{N}},\{\eta_i\}_{i \in N}$
    form a basis of $D(\Bc)$.
    In particular, we have 
    $\eta_i \in D(\Bc)$ for $i \in N$.
    
    Then, since  $Q(\Cc) = \prod_{i \in N}\alpha_i$, and by the multilinearity of the determinant and
    Theorem \ref{Thm_SaitosCriterion}, we conclude
    \begin{align*}
        \det( M(\theta_1,\ldots,\theta_\ell) ) 
        = \, &\frac{\det( M(\theta_1,\ldots,\theta_{\ell-p},\{\psi_j\}_{j \in \overline{N}},\{\alpha_i\eta_i\}_{i \in N}) )}{Q(\Cc)} \\
        &\in \, \KK^\times\frac{Q(\Ac')}{Q(\Cc)} = \KK^\times Q(\Bc).
    \end{align*}
    Consequently, since $\psi_j \in D(\Ac)  \subseteq D(\Bc)$, for $j \in \overline{N}$, all derivations $\theta_1,\ldots,\theta_\ell$ do belong to $D(\Bc)$ and thus they
    provide the desired basis of $D(\Bc)$, again by Theorem \ref{Thm_SaitosCriterion}.
\end{proof}

From the proof of Proposition \ref{Prop_TwoMATsteps} we record the following:
\begin{corollary}
    \label{Coro_TwoMATsteps}
    Let $\Ac'$, $\Ac$, $\Bc$, and $q$
    be as in Proposition \ref{Prop_TwoMATsteps}.
    Let $\Cc = \{\ker(\beta_1),\ldots,\ker(\beta_q)\}$.
    Then there is a basis $\theta_1,\ldots,\theta_{\ell-q},\eta_1,\ldots,\eta_q$ of $D(\Bc)$
    such that $\theta_1,\ldots,\theta_{\ell-q},\beta_1\eta_1,\ldots,\beta_q\eta_q$
    is a basis of $D(\Ac')$, $\deg{\theta_i} = e_i$, and $\theta_i \in D(\Ac)$ for $1\leq i \leq \ell-q$.
\end{corollary}

%%%%%%%%%%%%%%%%%%%%%%%%%%%%%%%%%%%%%%%%%%%%%%%%%%%%%%%%%%%%%%%%%%%%%%%%%%%%%%%%%%%%%%%

\subsection{Restrictions of MAT-free arrangements}
\label{SSec_ProofMainRes}

For this subsection, let $\Ac$ be an MAT-free arrangement with MAT-partition $\pi = (\pi_1|\cdots|\pi_n)$.
Recall from 
Definition \ref{Def_MATfree} 
that for $k \in \{1,\ldots,n\}$ 
 we set $\Ac_k := \bigcup_{i=1}^k \pi_i$
 and $\pi_j = \emptyset$ for $j>n$.
 
Our next result is used to maintain a certain part of a basis of the module of derivations of a subarrangement
and also later to construct a particular basis of $D(\Ac)$.

\begin{lemma}
    \label{Lem_ConsecMATstepsD}
    Fix $k \in \{1,\ldots,n\}$.
    If $\theta \in D(\Ac_k)$ is homogeneous with $\deg(\theta) < k$, then $\theta \in D(\Ac)$.
\end{lemma}

\begin{proof}
    For each of the MAT-steps $\Ac_{s+1} = \Ac_s \cup \pi_{s+1}$, $k\leq s \leq n-1$,
    we have $\deg(\theta) < k \leq s = |\Ac_s| - |(\Ac_s\cup \{H\})^H|$ 
    for each $H \in \pi_{s+1}$ (by condition (3) of Theorem \ref{Thm_MAT}). Hence  $\theta \in D(\Ac_{s+1})$, by Corollary \ref{Coro_degTheta}, 
    and ultimately we get $\theta \in D(\Ac)$.
\end{proof}

Our next result provides the right setting to apply Proposition \ref{Prop_FreeSubarrIneqRk}.
\begin{proposition}\label{Prop_BasisDMAT-Part}
    Fix $k \in \{1,\ldots,n\}$.
	For each $|\pi_{k+1}|\leq q \leq |\pi_k|$
    there are a $\Cc \subseteq \pi_k$ with $|\Cc|=q$ and
    a basis $\theta_1,\ldots,\theta_\ell$
    of $D(\Ac_k \setminus \Cc)$ 
    such that $\theta_1,\ldots,\theta_{\ell-q} \in D(\Ac)$.
\end{proposition}

\begin{proof}
	For this proof, set $p_j := |\pi_j|$ for all $j \geq 0$.
	
    First let $k=n$. 
    Then by Lemma \ref{Lem_BasisMATStep} applied to the MAT-step $\Ac = \Ac_{n-1} \cup \pi_n$, there is a basis $\theta_1,\ldots,\theta_{\ell-p_n},\eta_1,\ldots,\eta_{p_n}$
    of $D(\Ac_{n-1})$ such that the derivations $\theta_1,\ldots,\theta_{\ell-p_n}$, $\alpha_1\eta_1,\ldots,\alpha_{p_n}\eta_{p_n}$
    form a basis of $D(\Ac)$. 
    Further, for each $p_{n+1}=0\leq q \leq p_n$, set $\Cc := \{\ker(\alpha_{p_n-q+1}),\ldots,\ker(\alpha_{p_n})\} \subseteq \pi_n$, 
    and define $$\theta_{\ell-p_n+1} := \alpha_1\eta_1,\ldots,\theta_{\ell-q} :=\alpha_{p_n-q}\eta_{p_n-q} \in D(\Ac),$$ and 
    $\theta_{\ell-q+1} := \eta_{p_n-q+1},\ldots,\theta_{\ell} := \eta_{p_n}$.
    Then by yet another application of Lemma \ref{Lem_BasisMATStep}, we get that $\theta_1,\ldots,\theta_\ell$ form a basis of $D(\Ac_n\setminus \Cc)$
    such that $\theta_1,\ldots,\theta_{\ell-q} \in D(\Ac)$.
    %the statement readily follows from Lemma \ref{Lem_BasisMATStep}.    
    
    Now suppose $k<n$. 
    Then for the two MAT-steps $\Ac_{k} = \Ac_{k-1} \cup \pi_k$ and $\Ac_{k+1} = \Ac_k \cup \pi_{k+1}$
    we are precisely in the situation 
    of Proposition \ref{Prop_TwoMATsteps}.
    Suppose that $\exp(\Ac_{k}) = (e_1,\ldots,e_{\ell-p_{k}},k,\ldots,k)_\leq$.
    So by the statement of the lemma, for each $p_{k+1} \leq q \leq p_k$ there
    exists a $\Cc \subseteq \pi_k$ with $|\Cc|=q$ and a basis $\theta_1,\ldots,\theta_\ell$
    of $D(\Ac_k\setminus \Cc)$ 
    such that $\theta_i \in D(\Ac_{k+1})$ for $1\leq i \leq \ell-q$, 
    $\deg(\theta_i) = e_i \leq k$ for $1\leq i \leq \ell-p_k$,
    and $\deg(\theta_j) = k$ for $\ell-p_k+1 \leq j \leq \ell-q$.
    
    Finally, since $\deg(\theta_i) < k+1$ for $1\leq i \leq \ell-q$,
    we also have $\theta_i \in D(\Ac)$
    for $1\leq i \leq \ell-q$, by Lemma \ref{Lem_ConsecMATstepsD}.
\end{proof}

The following corollary to the preceding proposition is the key to the proof of Theorem \ref{Thm_MATRest}.

\begin{corollary}
    \label{Coro_KeyResMAT}
    Let $k, q$ and $\Cc$ be as in Proposition \ref{Prop_BasisDMAT-Part}.
    Let $X \in L(\Ac)$ with $\Ac_X \cap (\Ac_k \setminus \Cc) = \varnothing$.
    Then $\rk(X) \leq q$.    
\end{corollary}

\begin{proof}
    Owing to Proposition \ref{Prop_BasisDMAT-Part}, there is a basis $\theta_1,\ldots,\theta_\ell$
    of $D(\Ac_k \setminus \Cc)$ such that $\theta_1,\ldots,\theta_{\ell-q} \in D(\Ac)$.
    Since $X \in L(\Ac)$ with $\Ac_X \cap (\Ac_k \setminus \Cc) = \varnothing$, we readily obtain
    $\rk(X) \leq q$, according to Proposition \ref{Prop_FreeSubarrIneqRk}.
\end{proof}

Next we define a surjective map from a subarrangement to the restriction of the ambient arrangement.
\begin{lemma}
    \label{Lem_MATSurj}
    Let $k, q$ and $\Cc$ be as in Proposition \ref{Prop_BasisDMAT-Part}.
    Set $X := \cap_{H \in \Cc}H$.
    Then the map
    \begin{align*}
        \Ac_k \setminus \Cc &\to \Ac^X \\
        H' &\mapsto H'\cap X
    \end{align*}
    is surjective.
    In particular, $(\Ac_k)^X =  \Ac^X$.
\end{lemma}

\begin{proof}
    Set $\Bc := \Ac_k\setminus \Cc$.
    Let $Y = \tilde{H} \cap X \in \Ac^X$ for some $\tilde{H} \in \Ac$. In particular,
    $\rk(Y) = \rk(X)+1 = q +1 > q$. 
    Thus we have $\Ac_Y \cap \Bc \neq \varnothing$, by Corollary \ref{Coro_KeyResMAT}.
    Further, by the conditions (1) and (2) of Theorem \ref{Thm_MAT}, we have $\Bc_X = \varnothing$.
    So there is an $H' \in (\Ac_Y \setminus \Ac_X) \cap \Bc$
    such that $Y = H' \cap X$.
\end{proof}

Armed with the various results about MAT-steps from above, we are finally in a position to attack 
Theorem \ref{Thm_MATRest}.

\begin{theorem}\label{Thm_MATRest_sec}
    Let $\Ac$ be an MAT-free arrangement with MAT-partition $\pi = (\pi_1 | \cdots | \pi_n )$
    and exponents $\exp(\Ac) = (e_1,\ldots,e_\ell)_\leq$.
    Then for each $1 \leq k \leq n$ and each $|\pi_{k+1}| \leq q \leq |\pi_k|$ there is a $\Cc \subseteq \pi_k$ with $|\Cc|=q$
    such that for $X := \cap_{H \in \Cc}H$ 
    the restriction $\Ac^{X}$ is free with exponents $$\exp\left(\Ac^{X}\right) = (e_1,\ldots,e_{\ell-q})_\leq.$$

    Furthermore, there is a basis $\theta_1,\ldots,\theta_\ell$ of $D(\Ac)$
    such that $\theta^X_1, \ldots, \theta^X_{\ell-q}$ is
    a basis of $D(\Ac^{X})$.
\end{theorem}

\begin{proof}
    First note that $\Ac_k = \Ac_{k-1} \cup \pi_k$ is an MAT-step.
    Thanks to Corollary \ref{Coro_MAT-stepRestriction} the restriction $(\Ac_{k})^{X}$ is
    free with $\exp\left((\Ac_{k})^{X}\right) = (e_1,\ldots,e_{\ell-q})_\leq$.
    Now we also have $\Ac^{X} = (\Ac_{k})^{X}$, by Lemma \ref{Lem_MATSurj}, 
    and so the first statement follows.
    
    Now suppose $\Cc = \{\ker(\beta_1),\ldots,\ker(\beta_q)\}$.
    To derive the existence of the particular basis compatible with restriction,
    observe that there is a basis $\theta_1,\ldots,\theta_{\ell-q},\eta_1,\ldots,\eta_q$
    of $D(\Ac_{k}\setminus \Cc)$ with $\deg(\theta_i) = e_i$, $\theta_i \in D(\Ac_{k+1})$ for $1\leq i \leq \ell-q$
    so that $\theta_1,\ldots,\theta_{\ell-q},\beta_1\eta_1,\ldots,\beta_q\eta_q$ form a basis of $D(\Ac_k)$, owing to Corollary \ref{Coro_TwoMATsteps}. 
    Thanks to condition (1) in Theorem \ref{Thm_MAT} the linear forms $\beta_1,\ldots,\beta_q$
    are linearly independent. 
    Using the same arguments as the ones in the proof of \cite[Thm.~4.46]{OrTer92_Arr}
    successively along with a simple induction on $q$, the restrictions $\theta_i^X$ of the $\theta_i$ to $X$
    for $1\leq i \leq \ell-q$
    yield a basis of $D(\Ac_k^X) = D(\Ac^X)$.
    
    Considering the degrees of the $\theta_i$, 
    we have $\theta_1,\ldots,\theta_{\ell-q} \in D(\Ac)$, by Lemma \ref{Lem_ConsecMATstepsD}.
    Now thanks to Proposition \ref{Prop_ContinueBasisD}, we may extend $\theta_1,\ldots,\theta_{\ell-q}$
    to a complete basis $\theta_1,\ldots,\theta_{\ell-q},\theta_{\ell-q+1},\ldots,\theta_\ell$
    of $D(\Ac)$.
    This finishes the proof.
\end{proof}

\begin{corollary}[Theorem \ref{Thm_MATRest}]
    MAT-free arrangements are accurate.
\end{corollary}

\subsection{A generalization of Theorem \ref{Thm_MATRest}}
\label{SSec_MainResGen}

In this section, we record an extension of Theorem \ref{Thm_MATRest_sec}.
It can be applied to obtain the accuracy of not necessarily MAT-free arrangements (which we do in Section \ref{SSec_IdealShiCatalan}).
Instead of considering MAT-free arrangements
which are build up from the empty arrangement using MAT-steps,
we can more generally study arrangements which are build up
from some suitable free arrangement (which need not be empty) via MAT-steps.
We formalize this in the following manner.

Let $\Ac' = (\Ac', V)$ be a free arrangement with exponents 
$\exp(\Ac') = (e_1,\ldots,e_{\ell-p},e,\ldots,$ $e)_\leq \in \ZZ^{\ell}$. 
Let $\Bc = (\Bc, V)$ be an arrangement in $V$ disjoint from $\Ac$ and partitioned by
$\pi = (\pi_1|\cdots|\pi_n)$. 
Set $\Ac_0 := \Ac'$,
\[
    \Ac_k := \Ac'\dot{\cup} (\cup_{i=1}^k \pi_i),
\]
and $\pi_j = \emptyset$ for $j>n$.
Suppose that $\Ac_k \cup \pi_{k+1}$ is an MAT-step for $0\leq k \leq n-1$.
Then $\Ac = \Ac' \cup \Bc$ is free by a successive application of Theorem \ref{Thm_MAT}.
In this setting, all arguments from Sections \ref{SSec_ProofMainTechPrelim} and \ref{SSec_ProofMainRes} apply 
and Theorem \ref{Thm_MATRest_sec} directly generalizes
in the following fashion.

\begin{theorem}
    \label{Thm_MATRestGen}
    Let $\Ac = \Ac' \dot{\cup} \Bc$ be a free arrangement obtained from the free arrangement $\Ac'$
    through MAT-steps with exponents $\exp(\Ac) = (e_1,\ldots,e_\ell)_\leq$. 
    Suppose that $\pi = (\pi_1 | \cdots | \pi_n )$ is the corresponding
    ordered partition of $\Bc$. 
    Then for each $1 \leq k \leq n$ and each $|\pi_{k+1}| \leq q \leq |\pi_k|$ there is a $\Cc \subseteq \pi_k$ with $|\Cc|=q$
    such that for $X := \cap_{H \in \Cc}H$ 
    the restriction $\Ac^{X}$ is free with exponents 
    \[
        \exp\left(\Ac^{X}\right) = (e_1,\ldots,e_{\ell-q})_\leq.
    \]

    Furthermore, there is a basis $\theta_1,\ldots,\theta_\ell$ of $D(\Ac)$
    such that $\theta^X_1, \ldots, \theta^X_{\ell-q}$ is
    a basis of $D(\Ac^{X})$.
\end{theorem}

As a corollary to Theorem \ref{Thm_MATRestGen} we get the following sufficient condition
for accuracy.

\begin{corollary}
	\label{Coro_MATRestGen}
	Let $\Ac$, $\Bc$ and $\pi$ be as in Theorem \ref{Thm_MATRestGen}.
	If $|\pi_1| \geq \ell-2$ then $\Ac$ is accurate.
\end{corollary}

We conclude this section with an example 
which shows that the ordered partition $\pi = (\pi_1 | \cdots | \pi_n )$ of the subarrangement $\Bc \subseteq \Ac$
in Theorem \ref{Thm_MATRestGen}  
is itself not an MAT-partiton for $\Bc$ and $\Bc$ does not even need to be free, in general.

\begin{example}
	\label{Rem_BSubarr}
	Let $\Ac$ be the $3$-arrangement with defining polynomial
	\[
		Q(\Ac) = x_1x_2x_3(x_1-x_2)(x_1-x_3)(x_2-x_3)(x_2+x_3).
	\]
	In particular, $\Ac$ is an ideal subarrangement of the Weyl arrangement of type $B_3$, see Section \ref{SSec_WeylIdeal}.
	Then $\Ac' = \{\ker(x_1-x_2), \ker(x_2-x_3), \ker(x_3)\}$, $\Bc = \Ac \setminus \Ac'$
	and $\pi = (\pi_1,\pi_2)$ with
	\begin{align*}
		\pi_1 &= \{ \ker(x_1-x_3), \ker(x_2)\}, \\
		\pi_2 &= \{ \ker(x_2+x_3), \ker(x_1) \},
	\end{align*}
	apparently satisfy all the assumptions of Theorem \ref{Thm_MATRestGen}.
	But here, $\Bc$ is a generic $3$-arrangement with four hyperplanes and as such, it is not even free,
	e.g.\ by Terao's Factorization Theorem \cite{terao:freefactors}, as its characteristic polynomial 
	has non-integral roots.
\end{example}

%%%%%%%%%%%%%%%%%%%%%%%%%%%%%%%%%%%%%%%%%%%%%%%%%%%%%%%%%%%%%%%%%%%%%%%%%%%%%%%%%%%%%%%

\section{Applications}
\label{Sec_Appl}

In this section we apply our main Theorem \ref{Thm_MATRest_sec} and its generalization, Theorem \ref{Thm_MATRestGen}
to Weyl arrangements, their ideal subarrangements, extended Catalan arrangements and ideal-Shi arrangements.

\subsection{Weyl and Ideal arrangements}
\label{SSec_WeylIdeal}

For general information about Weyl groups and their root systems, see \cite{bourbaki:groupes}.

Let $W$ be a Weyl group acting as a reflection group on the real vector space $V = \RR^\ell$.
Let $\Phi := \Phi(W) \subseteq V^*$ be a (reduced) root system for $W$ and $\Phi^+ \subseteq \Phi$ a positive system
with simple roots $\Delta \subseteq \Phi^+$.
The \emph{rank} of $W$ respectively $\Phi$ is $\rk(W) := \rk(\Phi) := \dim \RR\Phi$.
We have $\Phi^+ = (\sum_{\alpha \in \Delta} \ZZ_{\geq 0} \alpha) \cap \Phi$, i.e.\
if $\beta \in \Phi^+$ then there are integers $n_\alpha \in \ZZ_{\geq0}$ such that
$\beta = \sum_{\alpha \in \Delta} n_\alpha \alpha$.
Then the \emph{height} of $\beta$ is defined by
$$
\h(\beta) := \sum_{\alpha\in\Delta}n_\alpha.
$$

The partial order $\leq$ on $\Phi^+$ is defined by
$$
    \beta \leq \gamma \, :\iff \, \gamma-\beta \in \sum_{\alpha \in \Delta} \ZZ_{\geq 0}\alpha.
$$

A subset $\Ic \subseteq \Phi^+$ is an \emph{ideal} if it is a (lower) order ideal 
in the poset $(\Phi^+,\leq)$, i.e.\ for $\alpha \in \Ic$ and $\beta \in \Phi^+$
with $\beta \leq \alpha$, we have $\beta \in \Ic$.

The \emph{Weyl arrangement} $\Ac(W)$
is the hyperplane arrangement in $V$ defined by 
$$
    \Ac(W) := \{\ker(\beta) \mid \beta \in \Phi^+ \}.
$$

\begin{definition}[{\cite{ABCHT16_FreeIdealWeyl}}]
	\label{DEF_ideal}	
    If $\Ic \subseteq \Phi^+$ is an order ideal then 
    \[
        \Ac_\Ic := \{ \ker(\beta) \mid \beta \in \Ic \} \subseteq \Ac(W)
    \]
    is called an 
    \emph{ideal (sub)arrangement}.
\end{definition}

We denote by $m_\Ic$ the maximal height of a root in $\Ic$.
For $1 \leq k \leq m_\Ic$,
let $$\pi_{k,\Ic} := \{ \ker(\alpha) \mid \alpha \in \Ic, \h(\alpha) = k \}$$ 
and let 
\[
    \pi_\Ic := (\pi_{1,\Ic}| \cdots | \pi_{m_\Ic,\Ic}),
\] %
be the \emph{root-height partition} of $\Ac_\Ic$. 
%Set $$p_{k,\Ic} := |\pi_{k,\Ic}|$$ for $1\leq k \leq m_\Ic$ and $p_{m_\Ic+1,\Ic} := 0$.
Set $\pi_{j,\Ic} = \emptyset$ for $j > m_\Ic$.

Next we recall the principal result from \cite[Thm.~1.1]{ABCHT16_FreeIdealWeyl} (Ideal-free Theorem) in our terminology, using the notation above.

\begin{theorem}
    \label{Thm_IdealFree}
    Let $\Ac_{\Ic} \subseteq \Ac(W)$ be an ideal subarrangement of the Weyl arrangement $\Ac(W)$
    for the order ideal $\Ic \subseteq \Phi^+$. 
    Then $\Ac_\Ic$ is MAT-free with MAT-partition
    $\pi_\Ic = (\pi_{1,\Ic} |\cdots| \pi_{m_\Ic,\Ic} )$ and exponents 
    \[
        \exp\left(\Ac_\Ic\right) = (e_1^{\Ic},\ldots,e_{\ell}^{\Ic}),
    \] 
    where
    \begin{align*}
        e_r^{\Ic} = |\{ j \mid |\pi_{j,\Ic}| \geq \ell-r+1\}|.
    \end{align*}
\end{theorem}

Combining Theorems \ref{Thm_MATRest_sec} and \ref{Thm_IdealFree} and using the notation from above,
we immediately obtain the following.

\begin{theorem}
    \label{Thm_HtRestIdeal_sec} %
    Let $\Ac = \Ac_{\Ic} \subseteq \Ac(W)$ be an ideal subarrangement of the Weyl arrangement $\Ac(W)$
    for the order ideal $\Ic \subseteq \Phi^+$. 
    
    Then for each $1\leq k \leq m_\Ic$ and  $|\pi_{k+1,\Ic}| \leq q \leq |\pi_{k,\Ic}|$
    there is a $\Cc \subseteq \pi_{k,\Ic}$ with $|\Cc| = q$
    such that for $X := \cap_{H \in \Cc}H$ 
    the restriction $\Ac^{X}$ is free with exponents 
    \[
        \exp\left(\Ac^{X}\right) = (e_1^{\Ic},\ldots,e_{\ell-q}^{\Ic}),
    \] 
    where
    \begin{align*}
        e_r^{\Ic} = |\{ j \mid |\pi_{j,\Ic}| \geq \ell-r+1\}|.
    \end{align*}
\end{theorem}

\begin{corollary}[Theorem \ref{Thm_HtRestIdeal}]
    Ideal arrangements are accurate.
\end{corollary}

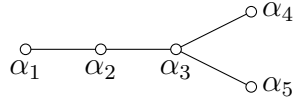
\begin{figure}
	
	\begin{center}
	\begin{tikzpicture}[scale=1.0]
		\node (1) at (0,0) {};
		\node (2) at (1,0) {};
		\node (3) at (2,0) {};
		\node (4) at (3,0.5) {};
		\node (5) at (3,-0.5) {};
		
		\draw (1) circle[radius=2pt];
		\draw (2) circle[radius=2pt];
		\draw (3) circle[radius=2pt];
		\draw (4) circle[radius=2pt];
		\draw (5) circle[radius=2pt];
		
		\draw [shorten <=2pt, shorten >=2pt] ($(1)$) -- ($(2)$);
		\draw [shorten <=2pt, shorten >=2pt] ($(2)$) -- ($(3)$);
		\draw [shorten <=2pt, shorten >=2pt] ($(3)$) -- ($(4)$);
		\draw [shorten <=2pt, shorten >=2pt] ($(3)$) -- ($(5)$);
		
		\node[below] at (1) {$\alpha_1$};
		\node[below] at (2) {$\alpha_2$};
		\node[below] at (3) {$\alpha_3$};
		\node[right] at (4) {$\alpha_4$};
		\node[right] at (5) {$\alpha_5$};			
	\end{tikzpicture}
	\end{center}
	\caption{The Dynkin diagram of $D_5$ with the corresponding numbering of the simple roots.}
	\label{Fig_ExDynD5}
	
\end{figure}

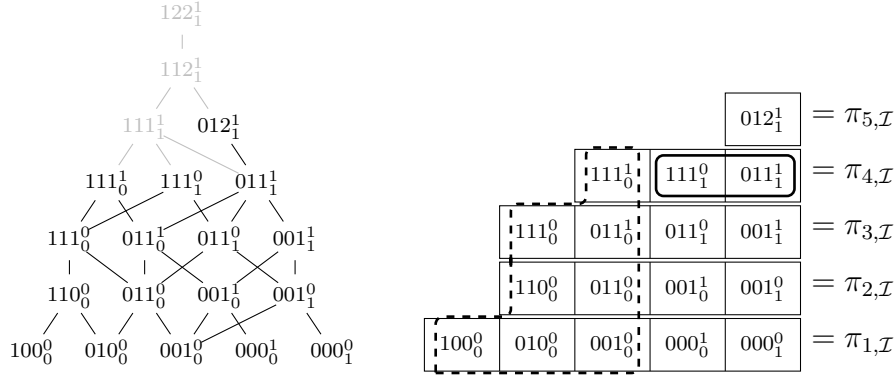
\begin{figure}
	
	\begin{center}
	\begin{tikzpicture}[scale=1.0]
		\node (5) at (1.5,0.75) {\tiny{$000^0_1$}};
		\node (4) at (0.5,0.75) {\tiny{$000^1_0$}};
		\node (3) at (-0.5,0.75) {\tiny{$001^0_0$}};
		\node (2) at (-1.5,0.75) {\tiny{$010^0_0$}};
		\node (1) at (-2.5,0.75) {\tiny{$100^0_0$}};
		
		\node (9) at (1.,1.5) {\tiny{$001^0_1$}};
		\node (8) at (0.,1.5) {\tiny{$001^1_0$}};
		\node (7) at (-1.,1.5) {\tiny{$011^0_0$}};
		\node (6) at (-2.,1.5) {\tiny{$110^0_0$}};
		\draw  [shorten <=8pt, shorten >=8pt] ($(9)$) -- ($(5)$); 
		\draw  [shorten <=8pt, shorten >=8pt] ($(9)$) -- ($(3)$); 
		\draw  [shorten <=8pt, shorten >=8pt] ($(8)$) -- ($(4)$); 
		\draw  [shorten <=8pt, shorten >=8pt] ($(8)$) -- ($(3)$); 
		\draw  [shorten <=8pt, shorten >=8pt] ($(7)$) -- ($(3)$); 
		\draw  [shorten <=8pt, shorten >=8pt] ($(7)$) -- ($(2)$); 
		\draw  [shorten <=8pt, shorten >=8pt] ($(6)$) -- ($(2)$); 
		\draw  [shorten <=8pt, shorten >=8pt] ($(6)$) -- ($(1)$); 
		
		\node (13) at (1.,2.25) {\tiny{$001^1_1$}};
		\node (12) at (0.,2.25) {\tiny{$011^0_1$}};
		\node (11) at (-1.,2.25) {\tiny{$011^1_0$}};
		\node (10) at (-2.,2.25) {\tiny{$111^0_0$}};
		\draw  [shorten <=8pt, shorten >=8pt] ($(13)$) -- ($(9)$); 
		\draw  [shorten <=8pt, shorten >=8pt] ($(13)$) -- ($(8)$); 
		\draw  [shorten <=8pt, shorten >=8pt] ($(12)$) -- ($(9)$); 
		\draw  [shorten <=8pt, shorten >=8pt] ($(12)$) -- ($(7)$); 
		\draw  [shorten <=8pt, shorten >=8pt] ($(11)$) -- ($(8)$); 
		\draw  [shorten <=8pt, shorten >=8pt] ($(11)$) -- ($(7)$); 
		\draw  [shorten <=8pt, shorten >=8pt] ($(10)$) -- ($(7)$); 
		\draw  [shorten <=8pt, shorten >=8pt] ($(10)$) -- ($(6)$);

		\node (16) at (0.5,3.) {\tiny{$011^1_1$}};
		\node (15) at (-0.5,3.) {\tiny{$111^0_1$}};
		\node (14) at (-1.5,3.) {\tiny{$111^1_0$}};		
		\draw  [shorten <=8pt, shorten >=8pt] ($(16)$) -- ($(13)$); 
		\draw  [shorten <=8pt, shorten >=8pt] ($(16)$) -- ($(12)$); 
		\draw  [shorten <=8pt, shorten >=8pt] ($(16)$) -- ($(11)$); 
		\draw  [shorten <=8pt, shorten >=8pt] ($(15)$) -- ($(12)$); 
		\draw  [shorten <=8pt, shorten >=8pt] ($(15)$) -- ($(10)$);
		\draw  [shorten <=8pt, shorten >=8pt] ($(14)$) -- ($(11)$); 
		\draw  [shorten <=8pt, shorten >=8pt] ($(14)$) -- ($(10)$);

		\node (18) at (0.,3.75) {\tiny{$012^1_1$}};
		\draw  [shorten <=8pt, shorten >=8pt] ($(18)$) -- ($(16)$); 
		
		{\color{lightgray}
			\node (17) at (-1.,3.75) {\tiny{$111^1_1$}};
			\draw  [shorten <=8pt, shorten >=8pt] ($(17)$) -- ($(16)$); 
			\draw  [shorten <=8pt, shorten >=8pt] ($(17)$) -- ($(15)$); 
			\draw  [shorten <=8pt, shorten >=8pt] ($(17)$) -- ($(14)$); 
			
			\node (19) at (-0.5,4.5) {\tiny{$112^1_1$}};
			\draw  [shorten <=8pt, shorten >=8pt] ($(19)$) -- ($(18)$); 
			\draw  [shorten <=8pt, shorten >=8pt] ($(19)$) -- ($(17)$);

			\node (20) at (-0.5,5.25) {\tiny{$122^1_1$}};
			\draw  [shorten <=8pt, shorten >=8pt] ($(20)$) -- ($(19)$); 
		}		
	\end{tikzpicture}
	\hspace{0.5cm}
	\begin{tikzpicture}[scale=1.0]
		\node (5) at (1.5,0.75) {\tiny{$000^0_1$}};
		\node (4) at (0.5,0.75) {\tiny{$000^1_0$}};
		\node (3) at (-0.5,0.75) {\tiny{$001^0_0$}};
		\node (2) at (-1.5,0.75) {\tiny{$010^0_0$}};
		\node (1) at (-2.5,0.75) {\tiny{$100^0_0$}};
		
		\node (9) at (1.5,1.5) {\tiny{$001^0_1$}};
		\node (8) at (0.5,1.5) {\tiny{$001^1_0$}};
		\node (7) at (-0.5,1.5) {\tiny{$011^0_0$}};
		\node (6) at (-1.5,1.5) {\tiny{$110^0_0$}};
		
		\node (13) at (1.5,2.25) {\tiny{$001^1_1$}};
		\node (12) at (0.5,2.25) {\tiny{$011^0_1$}};
		\node (11) at (-0.5,2.25) {\tiny{$011^1_0$}};
		\node (10) at (-1.5,2.25) {\tiny{$111^0_0$}};
	
		\node (16) at (1.5,3.) {\tiny{$011^1_1$}};
		\node (15) at (0.5,3.) {\tiny{$111^0_1$}};
		\draw[rounded corners=0.1cm, line width = 1pt] ($(15)-(0.42,0.27)$) rectangle ($(16)+(0.42,0.27)$);

		\node (14) at (-0.5,3.) {\tiny{$111^1_0$}};		
		
		\node (18) at (1.5,3.75) {\tiny{$012^1_1$}};
		
		\draw ($(1)-(0.5,0.35)$) rectangle ($(5)+(0.5,0.35)$);
		\draw ($(1)+(0.5,0.35)$) -- ($(1)+(0.5,-0.35)$);
		\draw ($(2)+(0.5,0.35)$) -- ($(2)+(0.5,-0.35)$);
		\draw ($(3)+(0.5,0.35)$) -- ($(3)+(0.5,-0.35)$);
		\draw ($(4)+(0.5,0.35)$) -- ($(4)+(0.5,-0.35)$);
		\node[right]  at (2.,0.75) {$=\pi_{1,\Ic}$};
		
		\draw ($(6)-(0.5,0.35)$) rectangle ($(9)+(0.5,0.35)$);
		\draw ($(6)+(0.5,0.35)$) -- ($(6)+(0.5,-0.35)$);
		\draw ($(7)+(0.5,0.35)$) -- ($(7)+(0.5,-0.35)$);
		\draw ($(8)+(0.5,0.35)$) -- ($(8)+(0.5,-0.35)$);
		\node[right]  at (2.,1.5) {$=\pi_{2,\Ic}$};
		
		\draw ($(10)-(0.5,0.35)$) rectangle ($(13)+(0.5,0.35)$);
		\draw ($(10)+(0.5,0.35)$) -- ($(10)+(0.5,-0.35)$);
		\draw ($(11)+(0.5,0.35)$) -- ($(11)+(0.5,-0.35)$);
		\draw ($(12)+(0.5,0.35)$) -- ($(12)+(0.5,-0.35)$);
		\node[right]  at (2.,2.25) {$=\pi_{3,\Ic}$};
		
		\draw ($(14)-(0.5,0.35)$) rectangle ($(16)+(0.5,0.35)$);
		\draw ($(14)+(0.5,0.35)$) -- ($(14)+(0.5,-0.35)$);
		\draw ($(15)+(0.5,0.35)$) -- ($(15)+(0.5,-0.35)$);
		\node[right]  at (2.,3.) {$=\pi_{4,\Ic}$};
		
		\draw ($(18)-(0.5,0.35)$) rectangle ($(18)+(0.5,0.35)$);
		\node[right]  at (2.,3.75) {$=\pi_{5,\Ic}$};
		
		\draw[rounded corners=0.1cm, line width = 1.0pt, dashed] 
		($(1)+(-0.35,-0.375)$) -- ($(3)+(0.35,-0.375)$) -- ($(14)+(0.35,0.375)$)
		-- ($(14)+(-0.35,0.375)$) -- ($(14)+(-0.35,-0.375)$) -- ($(10)+(-0.35,0.375)$)
		-- ($(10)+(-0.35,-0.375)$)-- ($(6)+(-0.35,0.375)$)-- ($(6)+(-0.35,-0.375)$)
		-- ($(1)+(0.35,0.375)$)-- ($(1)+(-0.35,0.375)$)-- ($(1)+(-0.35,-0.375)$);
		
	\end{tikzpicture}
	\end{center}
	\caption{The root-poset of type $D_5$, the height partition of the ideal $\Ic$, 
		and a restriction of the ideal subarrangement $\Ac_\Ic \subseteq \Ac(D_5)$.} 
	\label{Fig_ExIdealRes}

\end{figure}

\begin{example}
	Let $\Phi$ be the root system of type $D_5$ and 
	write ${c_1c_2c_3}^{c_4}_{c_5}$ for $\beta = \sum_i c_i\alpha_i \in \Phi^+$
	with the numbering of the simple roots according to the Dynkin diagram displayed in Figure \ref{Fig_ExDynD5}.
	Let $\Ic := \{ \beta \in \Phi^+ \mid \beta \leq 012_1^1,$ or $\beta \leq 111_0^1,$ or $\beta \leq 111_1^0 \}$
	be the ideal as shown in Figure \ref{Fig_ExIdealRes}.
	
	Then $m_\Ic = 0+1+2+1+1 = 5$ and $\Ac_\Ic$ is MAT-free with MAT-partition $(\pi_{1,\Ic}|\cdots|\pi_{5,\Ic})$
	and exponents $\exp(\Ac_\Ic) = (1,3,4,4,5)$.

	Now, for $Y := \ker(010^0_0) \cap \ker(001^0_0)$ the characteristic polynomial 
	of the restriction $\Ac_\Ic^Y$ does not factor over $\ZZ$ and consequently, the restriction $\Ac_\Ic^Y$
	is not free by Terao's Factorization Theorem \cite{terao:freefactors}.
	
	But by Theorem \ref{Thm_HtRestIdeal_sec} 
	for e.g.\ $\Cc = \{\ker(111^0_1), \ker(011^1_1)\} \subseteq \pi_{4,\Ic}$
	and $X = \cap_{H \in \Cc}H = \ker(111^0_1) \cap \ker(011^1_1)$,
	the restriction $\Ac_\Ic^{X}$ is free with exponents $\exp(\Ac_\Ic^{X}) = (1,3,4)$.
	Note that the exponents can be easily read off 	the enclosed part (dashed line) of the partition diagram
	as illustrated in Figure \ref{Fig_ExIdealRes}.
\end{example}

The preceding example shows that Theorem \ref{Thm_HtRestIdeal_sec} is probably the strongest
general statement one can give concerning the freeness of restrictions of ideal arrangements and their exponents.

%%%%%%%%%%%%%%%%%%%%%%%%%%%%%%%%%%%%%%%%%%%%%%%%%%%%%%%%%%%%%%%%%%%%%%%%%%%%%%%%%%%%%%%

\subsection{Ideal-Shi and extended Catalan arrangements}
\label{SSec_IdealShiCatalan}

Let $W$, $V$ and $\Delta \subseteq \Phi^+ \subseteq \Phi \subseteq V^*$ be as in Section \ref{SSec_WeylIdeal}.
Let $$h := m_{\Phi^+}+1$$ be the \emph{Coxeter number} of $W$.
Embed $V \subseteq V' = \RR^{\ell+1}$ and let $z \in (V')^* \setminus \{0\}$
such that $V = \ker(z) =: H_z$, i.e.\ $z$ corresponds to the $(\ell+1)$st coordinate.
For $j \in \ZZ$ define $$H_\alpha^j := \ker(\alpha-jz).$$

The combinatorics of deformations of Weyl arrangements was first studied by Athanasiadis in \cite{Ath96_CharPolySubspaceArr}, among them the so called extended Shi arrangements.
We are interested in the following generalization, investigated by Abe and Terao in \cite{AbeTer15_IdealShi}, the so called ideal-Shi arrangements.

\begin{definition}
    \label{Def_IdealShiCatalan}
    For $k \in \ZZ_{>0}$ the \emph{extended Shi arrangements}
    are defined as
    \[
         \Shi^k := \{ H_\alpha^j \mid -k+1\leq j \leq k \} \cup \{H_z \}.
    \]
    
    For $\Ic \subseteq \Phi^+$ an ideal,  
    the \emph{ideal-Shi arrangements}
    are defined as
    \[
        \Shi_\Ic^k := \Shi^k \cup  \{H_\beta^{-k} \mid \beta \in \Ic \}.
    \]
    
    In the special case when $\Ic = \Phi^+$ 
    we obtain the \emph{extended Catalan  arrangements} $$\Cat^k := \Shi_{\Phi^+}^k.$$
\end{definition}

In \cite{Yos04_CharaktFree}, Yoshinaga proved the following remarkable theorem, confirming a conjecture
by Edelman and Reiner \cite[Conj.~3.3]{ER96_FreeArrRhombicTilings}.
\begin{theorem}[{\cite[Thm.~1.2]{Yos04_CharaktFree}}]
    Let $k \in \ZZ_{>0}$. 
    Then the extended Shi arrangement $\Shi^k$ is free with exponents
    \[
        \exp\left(\Shi^k\right) = (1,hk,\ldots,hk) \in \ZZ^{\ell+1},
    \]
    and the extended Catalan arrangement $\Cat^k$ is free with exponents
    \[
        \exp\left(\Cat^k\right) = (1,hk+e_1,\ldots,hk+e_\ell),
    \]
    where $(e_1,\ldots,e_\ell) = \exp(\Ac(W))$.
\end{theorem}

We need the following special case of a result by Abe and Terao (cf.~\cite[Thm.~1.6]{AbeTer15_IdealShi}).
\begin{proposition}
    \label{Prop_Shi-SimplRoots}
        Let $\Sigma \subseteq \Delta$ be a subset of the simple roots of $\Phi^+$.
        Then the arrangement 
        \[
            \Shi^k_{-\Sigma} := \Shi^k \setminus \{H^k_\alpha \mid \alpha \in \Sigma\}
        \]
        is free with
        \[
            \exp\left(\Shi^k_{-\Sigma}\right) = (1,hk-1,\ldots,hk-1,hk,\ldots,hk),
        \] 
        where the multiplicity of the exponent $hk-1$ is $|\Sigma|$.
\end{proposition}

Next we observe that in the case $\Sigma = \Delta$, rejoining the hyperplanes 
$H^k_\alpha$ for $\alpha \in \Delta$
to $\Shi^k_{-\Delta}$ is an MAT-step.
\begin{lemma}
    \label{Lem_Shi-SimplRootMATStep}
    The addition $\Shi^k = \Shi^k_{-\Delta}\dot{\cup}\{H^k_\alpha \mid \alpha \in \Delta\}$ is an MAT-step.
\end{lemma}
\begin{proof}
    We have to verify conditions (1)--(3) from Theorem \ref{Thm_MAT}.
    Condition (1) is clear, since the simple roots $\Delta$  are linearly independent and so are the
    linear forms $\alpha-kz$ for $\alpha \in \Delta$.
    Condition (3) follows from Propositions \ref{Prop_AddDelRestr} and \ref{Prop_Shi-SimplRoots} and  
    the fact that the sum of the exponents of a free arrangement coincides with its cardinality, cf.~\cite[Thm.~4.23]{OrTer92_Arr}.
    
    Finally, we need to show condition (2), i.e.\ 
    that $X = \cap_{\alpha\in\Delta}H^k_\alpha \nsubseteq H$ for any $H = H^j_\beta \in \Shi^k_{-\Delta}$. For a contradiction, 
    we assume the contrary.
    That is, there are an index $-k+1\leq j \leq k$ and a positive root $\beta \in \Phi^+$ such that
    \[
        \beta-jz =  \sum_{\alpha \in \Delta} n_\alpha(\alpha-kz)
    \]
    for suitable $n_\alpha$.
    But then $\beta = \sum_{\alpha \in \Delta} n_\alpha \alpha$, so $n_\alpha \in \ZZ_{\geq 0}$ for all $\alpha \in \Delta$
    and $n_\alpha \geq 1$ for some $\alpha \in \Delta$ since $\beta$ is a positive root.
    Therefore, 
    \[
        j = \sum_{\alpha \in \Delta}n_\alpha k \geq k.
    \]
    Hence we must have $j=k$ and thus $\beta$ belongs to  $\Phi^+\setminus\Delta$.
    Consequently, in the expression $\beta = \sum_{\alpha \in \Delta} n_\alpha \alpha$ at least two coefficients 
    must be positive. 
    But then $k=j \geq 2k$, which is absurd, as $k>0$ in Definition \ref{Def_IdealShiCatalan}.
\end{proof}

In \cite[Sec.~5]{AbeTer18_MultAddDelRes}, Abe and Terao showed that an ideal-Shi arrangement
and in particular an extended Catalan arrangement can be constructed
via MAT-steps starting from the corresponding extended Shi arrangement.
Hence we are exactly in the setting of Theorem \ref{Thm_MATRestGen}
with $\Ac' = \Shi^k_{-\Delta}$, $\Bc = \{H^k_\alpha \mid \alpha \in \Delta\}\cup\{ H_{\beta}^{-k} \mid \beta \in \Ic \}$
and
\[
    \pi = \pi_\Ic^k := \left(\pi_{1,\Ic}^k| \cdots | \pi_{m_\Ic+1,\Ic}^k\right),
\]
where 
\begin{align*}
    \pi_{1,\Ic}^k = \, &\{H^k_\alpha \mid \alpha \in \Delta\}, \text{ and } \\
    \pi_{t+1,\Ic}^k := \, &\{ H_{\beta}^{-k} \mid \beta \in \Ic, \h(\beta) = t \},
\end{align*}
for $1 \leq t \leq m_\Ic$.
%Set $p_{t,\Ic} := |\pi_{t,\Ic}^k|$ for $1\leq t \leq m_\Ic+1$ and $p_{m_\Ic+2,\Ic} := 0$.
Set $\pi_{j,\Ic}^k = \emptyset$ for $j > m_\Ic+1$.
Then the arguments from the proofs in \cite[Sec.~5]{AbeTer18_MultAddDelRes}, 
together with Lemma \ref{Lem_Shi-SimplRootMATStep} and Theorem \ref{Thm_MATRestGen} readily imply the following.

\begin{theorem}
    \label{Thm_ResShiCat} %
    Let $\Shi_\Ic^k$ be an ideal-Shi arrangement
    for the order ideal $\Ic \subseteq \Phi^+$. 
    Then for each $1 \leq t \leq m_\Ic+1$ and each $|\pi_{t+1,\Ic}^k| \leq q \leq |\pi_{t,\Ic}^k|$
    there is a $\Cc \subseteq \pi_{t,\Ic}^k$ with $|\Cc| = q$
    such that for $X := \cap_{H \in \Cc}H$ 
    the restriction $\left(\Shi_\Ic^k\right)^{X}$ is free with exponents 
    \[
        \exp\left(\left(\Shi_\Ic^k\right)^{X}\right) = \left(1,hk+e_1^{\Ic},\ldots,hk+e_{\ell-q}^{\Ic}\right)_\leq,
    \] 
    where
    \begin{align*}
        e_r^{\Ic} = |\{ j \mid |\pi_{{j+1},\Ic}^k| \geq \ell-r+1\}|.
    \end{align*}
\end{theorem}

We close this section with the following special case of Theorem \ref{Thm_ResShiCat}.

\begin{corollary}[Theorem \ref{Thm_ResShiCatTF}]
    Extended Shi arrangements $\Shi^k$ and ideal-Shi arrangements $\Shi^k_\Ic$ are accurate.
    In particular, extended Catalan arrangements $\Cat^k$ are accurate.
\end{corollary}

%%%%%%%%%%%%%%%%%%%%%%%%%%%%%%%%%%%%%%%%%%%%%%%%%%%%%%%%%%%%%%%%%%%%%%%%%%%%%%%%%%%%%%%
%%%%%%%%%%%%%%%%%%%%%%%%%%%%%%%%%%%%%%%%%%%%%%%%%%%%%%%%%%%%%%%%%%%%%%%%%%%%%%%%%%%%%%%

\section{Complements and Examples}

In our final section we discuss other notions of freeness in relation to 
accuracy. 

Further we also study accurate graphic arrangements. 
Here we show that accuracy is neither preserved under taking localizations nor under restrictions.  

Finally, in view of Theorem \ref{thm:otsconj}, we consider accuracy among complex reflection arrangements and their restrictions.

\subsection{Almost accurate arrangements}
\label{ssect:divfree}

We begin by weakening the property from Definition \ref{Def_TF} by dropping the condition that the exponents of the free
restriction $\Ac^{X}$ match those of $\Ac$ ordered by size.

\begin{definition}
	\label{Def_aTF}
	An arrangement $\Ac$ is said to be \emph{almost accurate}
	if $\Ac$ is free with exponents $\exp(\Ac) = (e_1, e_2, \ldots, e_\ell)$ 
	and for every $1 \leq d \leq \ell$ there exists an intersection $X$ of hyperplanes  from  
	$\Ac$ of dimension $d$ such that $\Ac^{X}$ is free with $\exp(\Ac^{X}) \subseteq \exp(\Ac)$.
\end{definition}

Clearly, if $\Ac$ is accurate then it is almost accurate. 
In particular, by Theorem \ref{Thm_MATRest},
MAT-free arrangements are almost accurate.
However, the converse is false, see Example \ref{Ex_SSbnTF}. 

Note that a product of arrangements is almost accurate if and only if each factor is almost accurate, by \cite[Prop.~2.14, Prop.~4.28]{OrTer92_Arr}.

\bigskip

Next we indicate that there are several natural classes of
free arrangements that are almost accurate.

For an arrangement $\Ac$ we denote its characteristic polynomial by $\chi(\Ac,t)$, see  \cite[Sec.~2.3]{OrTer92_Arr}.
Next we recall the key result from \cite{abe:divfree}.

\begin{theorem}[{\cite[Thm.~1.1]{abe:divfree}}]
	\label{thm:abe-div}
	Suppose there is a hyperplane $H$ in $\Ac$ such that the restriction 
	$\Ac^H$ is free and that 
	$\chi(\Ac^H,t)$ divides $\chi(\Ac,t)$.
	Then $\Ac$ is free.
\end{theorem}

Theorem \ref{thm:abe-div} can be viewed as a strengthening of the addition
part of Terao's Theorem \cite{Terao1980_FreeI} (\cite[Thm.~4.51]{OrTer92_Arr}). An iterative  
application leads to the class of divisionally free arrangements:

\begin{definition}
	[{\cite[Def.~1.5]{abe:divfree}}]
	\label{def:divfree}
	An arrangement $\Ac$ is called 
	\emph{divisionally free} if either $\ell \le 2$, $\Ac = \varnothing_\ell$, or
	else there is a flag of subspaces $X_i$ of rank $i$ in $L(\Ac)$
	\[
	X_0 = V \supset X_1 \supset X_2 \supset \cdots \supset X_{\ell -2}
	\]
	so that $\chi(\Ac^{X_i},t)$ divides $\chi(\Ac^{X_{i-1}},t)$, for 
	$i = 1, \ldots, \ell-2$.
\end{definition}

It is clear from Theorem \ref{thm:abe-div}, Definition \ref{def:divfree}, and the fact that rank $2$ arrangements are free, that 
divisionally free arrangements are almost accurate. However, the converse is false, see Example \ref{ex:star-but-not-divfree}.
Nevertheless, since supersolvable arrangements are inductively free (\cite[Thm.~4.58]{OrTer92_Arr}) and the latter are divisionally free (\cite[Thm.~1.6]{abe:divfree}), there is an abundance of arrangements that are almost accurate.

\bigskip

It was already observed in \cite{CunMue19_MATfree}
that MAT-freeness and divisional freeness  are  
closely related. Thus, in view of 
Theorem \ref{Thm_MATRest}, it is natural to 
investigate the relation between divisional freeness and  accuracy.
However, Examples \ref{Ex_SSbnTF} and   \ref{ex:star-but-not-divfree} below do show that there are 
supersolvable (whence divisionally free) arrangements that are not accurate and vice versa.

\begin{example}
	\label{Ex_SSbnTF}
	Let $\Ac = \Ac(14,1)$ be the supersolvable simplicial arrangement of rank $3$ with $14$ hyperplanes,
	cf.~\cite{Grue09_SimplArr}, \cite{CunMue17_SuperSimpl}. 
	We have $\Ac = \{H_1,\ldots,H_7,K_1,\ldots,K_7 \}$, $X = \cap_{i=1}^7 H_i$ is modular
	of rank $2$, and $X \nsubseteq K_j$ for all $1\leq j \leq 7$.
	Consider $\Bc := \Ac \setminus \{K_1,K_2\}$. Since $\Bc$ is obtained from $\Ac$ by removing 
	hyperplanes away from the modular element $X$, $\Bc$ is still supersolvable
	and, thanks to \cite[Thm.~4.58]{OrTer92_Arr},	$\Bc$ is 
	 also free with exponents
	$\exp(\Bc) = (1,5,6)$. 
	Using the explicit description of $\Ac$ from  \cite{CunMue17_SuperSimpl}, 
	one easily checks that $\exp(\Bc^H)$ belongs to  $\{(1,3),(1,4),(1,6)\}$ for $H \in \Bc$.
	Consequently, $\Bc$ is not accurate.
\end{example}

\begin{example}
	\label{ex:star-but-not-divfree}
	In \cite{HogeRoehrle19_ConjAbeAF} a free arrangement $\Dc$ is constructed within a rank 5 restriction of the Weyl arrangement of type $E_7$ which is not divisionally free with exponents $(1,5,5,5,5)$ and defining polynomial
    \begin{align*}
        Q(\Dc) = \, &x_2(x_1+x_3-x_5)(2x_1+x_2+x_3)(2x_1+x_2+2x_3+x_4-x_5)\\
            &x_5(x_1+x_3)(x_2+x_5)(2x_1+x_2+2x_3+x_4)(2x_1+x_3-x_5)\\
            &(2x_1+2x_2+2x_3+x_4)(x_2+x_3+x_4)(x_1+x_2+x_3+x_4)\\
            &(x_3+x_4)(x_1+x_2+x_3)x_1(x_1+x_3+x_4)(2x_1+x_2+x_3-x_5)\\
            &(x_2+x_3+x_4+x_5)(x_1-x_5)(x_1-x_4-x_5)x_4.
    \end{align*}
	
	One can check that $\Dc $ is still accurate:
	Only the restriction to $H =\ker(x_4)$ has exponents $(1,5,5,5)$ and only the restrictions to
	\begin{align*}
        X_1 & = \ker(2x_1+x_2+x_3) \cap \ker(2x_1+x_2+2x_3+x_4-x_5), \\
        X_2 & = \ker(x_2+x_5) \cap \ker(x_2+x_3+x_4)
	\end{align*}
	are free with exponents $\exp(\Dc^{X_1}) = \exp(\Dc ^{X_2}) = (1,5,5)$. 
	However, neither of those flats is contained in $H = \ker(x_4)$.
	Note further that 
	$ Y_1 = X_1 \cap H$ and also $Y_2 = X_2 \cap H$ are rank 3 flats with 
	$\exp(\Dc^{Y_1}) = \exp(\Dc^{Y_2}) = (1,5)$. 
    In particular, the lack of a suitable rank $2$ flat lying between $Y_1$ (or $Y_2$)
	and $H$ prevents $\Dc$ from being divisionally free.
	
	It is also easily seen that the free but non-divisionally free rank $7$ arrangement
	$\Bc$ constructed in \cite{HogeRoehrle19_ConjAbeAF} as a certain subarrangement of the Weyl arrangement of type $E_7$ is also still accurate.	
	
	In \cite[\S 6]{CunMue19_MATfree}, Cuntz and M\"ucksch checked that both $\Dc$ and $\Bc$ fail to be MAT-free.
\end{example}

\begin{remark}
	\label{rem:AbeTeroTran}
	In the recent paper  \cite{AbeTerTran20_A_12Restr}, Abe, Terao and Tran
	derive the freeness of all restrictions $\Ac(W)^X$ of Weyl arrangements $\Ac(W)$ to intersections $X$ where $\Ac(W)_X/X \cong \Ac(A_1^2)$.
	\cite[Thm.~1.5]{AbeTerTran20_A_12Restr} asserts that the exponents of such a restriction $\Ac(W)^X$ are
	either given by removing the two highest exponents of $\Ac(W)$ or by removing the highest exponent and the middle exponent $h/2$,
	where $h$ is the Coxeter number of $W$.
	Consequently, together with Corollary \ref{Coro_HighestRootWeyl}, their result
	implies the divisional freeness and hence also the almost accuracy of Weyl arrangements of rank at most $4$. 
\end{remark}

\subsection{Graphic arrangements}
\label{ssect:graph}

In this section we examine accuracy among free graphic arrangements.
For basics on the latter, we refer to \cite[Sec.~2.4]{OrTer92_Arr}.

Our next example shows that also free graphic arrangements
which are always supersolvable and come from chordal graphs (cf.~\cite[Sec.~3]{ER94_FreeAB}) may fail to be accurate in general.

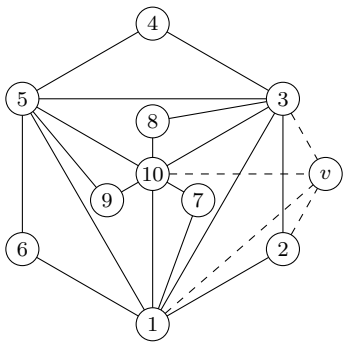
\begin{figure}[ht!b]
	\def\sc {1.0}
	\begin{tikzpicture}[scale=\sc]
	% Vertices 1--10
	\node (10) at (0.0,0.0) {\tiny $10$};
	\draw[black] (10) circle[radius=6.25pt];
	\node (1) at (0.0,-2.) {\tiny $1$};
	\draw[black] (1) circle[radius=6.25pt];
	\node (2) at (1.7320508075688774,-1.) {\tiny $2$};
	\draw[black] (2) circle[radius=6.25pt];
	\node (3) at (1.7320508075688774,1.) {\tiny $3$};
	\draw[black] (3) circle[radius=6.25pt];
	\node (4) at (0.0,2.) {\tiny $4$};
	\draw[black] (4) circle[radius=6.25pt];
	\node (5) at (-1.7320508075688774,1.) {\tiny $5$};
	\draw[black] (5) circle[radius=6.25pt];
	\node (6) at (-1.7320508075688774,-1.) {\tiny $6$};
	\draw[black] (6) circle[radius=6.25pt];
	\node (7) at ($0.35*(2)-0.35*(10)$) {\tiny $7$};
	\draw[black] (7) circle[radius=6.25pt];
	\node (8) at ($0.35*(4)-0.35*(10)$) {\tiny $8$};
	\draw[black] (8) circle[radius=6.25pt];
	\node (9) at ($0.35*(6)-0.35*(10)$) {\tiny $9$};
	\draw[black] (9) circle[radius=6.25pt];
	
	% Edges #=18
	\draw[shorten <=6.25pt, shorten >=6.25pt] ($(1)$) -- ($(2)$);
	\draw[shorten <=6.25pt, shorten >=6.25pt] ($(2)$) -- ($(3)$);
	\draw[shorten <=6.25pt, shorten >=6.25pt] ($(3)$) -- ($(4)$);
	\draw[shorten <=6.25pt, shorten >=6.25pt] ($(4)$) -- ($(5)$);
	\draw[shorten <=6.25pt, shorten >=6.25pt] ($(5)$) -- ($(6)$);
	\draw[shorten <=6.25pt, shorten >=6.25pt] ($(6)$) -- ($(1)$);
	\draw[shorten <=6.25pt, shorten >=6.25pt] ($(1)$) -- ($(10)$);
	\draw[shorten <=6.25pt, shorten >=6.25pt] ($(3)$) -- ($(10)$);
	\draw[shorten <=6.25pt, shorten >=6.25pt] ($(5)$) -- ($(10)$);
	\draw[shorten <=6.25pt, shorten >=6.25pt] ($(1)$) -- ($(3)$);
	\draw[shorten <=6.25pt, shorten >=6.25pt] ($(3)$) -- ($(5)$);
	\draw[shorten <=6.25pt, shorten >=6.25pt] ($(5)$) -- ($(1)$);
	
	\draw[shorten <=6.25pt, shorten >=6.25pt] ($(3)$) -- ($(8)$);
	\draw[shorten <=6.25pt, shorten >=6.25pt] ($(8)$) -- ($(10)$);
	\draw[shorten <=6.25pt, shorten >=6.25pt] ($(5)$) -- ($(9)$);
	\draw[shorten <=6.25pt, shorten >=6.25pt] ($(9)$) -- ($(10)$);
	\draw[shorten <=6.25pt, shorten >=6.25pt] ($(1)$) -- ($(7)$);
	\draw[shorten <=6.25pt, shorten >=6.25pt] ($(7)$) -- ($(10)$);
	
	\node (11) at (2.3,0.0) {\tiny $v$};
	\draw[black] (11) circle[radius=6.25pt];
	\draw[shorten <=6.25pt, shorten >=6.25pt, dashed] ($(1)$) -- ($(11)$);
	\draw[shorten <=6.25pt, shorten >=6.25pt, dashed] ($(2)$) -- ($(11)$);
	\draw[shorten <=6.25pt, shorten >=6.25pt, dashed] ($(3)$) -- ($(11)$);
	\draw[shorten <=6.25pt, shorten >=6.25pt, dashed] ($(10)$) -- ($(11)$);
	
	\end{tikzpicture}    
	\caption{A chordal graph $\Gg$ giving rise to a graphic arrangement which is free, but not accurate,
		and an extension $\Gg'$ by one vertex $v$ resulting in a graphic arrangement which is accurate.}
	\label{Fig_GraphNotAcc}
\end{figure}

\begin{example}
    \label{Ex_GraphNotAcc}
    Let $\Gg = (\Vg,\Eg)$ be the graph with vertex set $\Vg = \{1,\ldots, 10\}$ 
    and $18$ edges $\Eg$,  
    as shown in Figure \ref{Fig_GraphNotAcc}.
    It is easily seen that $(2,4,6,7,8,9,1,3,5,10)$ is a
    vertex elimination order for $\Gg$ (cf.~\cite[Sec.~3]{ER94_FreeAB}).
    Then by \cite[Thm.~3.3]{ER94_FreeAB}, the graphic arrangement $\Ac=\Ac(\Gg) = \{\ker(x_i-x_j) \mid (i,j) \in \Eg\}$
    is free with exponents  $\exp(\Ac) = (0,1,2,\ldots,$ $2,3) \in \ZZ^{10}$.
    From Figure \ref{Fig_GraphNotAcc} it is evident that up to symmetry there are only
    two different kinds of hyperplanes in $\Ac$.
    The hyperplane $H = \ker(x_1-x_2)$ is a representative of the first type and
    the hyperplane $K = \ker(x_1-x_3)$ is a representative of the second type.
    We easily see that 
    \begin{align*}
        \exp(\Ac^H) &= (0,1,2,\ldots,2,3) \in \ZZ^9, \\
        \exp(\Ac^K) &= (0,1,1,2,\ldots,2) \in \ZZ^9.
    \end{align*}
    Consequently, the graphic arrangement $\Ac$ is not accurate, in particular, it is not MAT-free, by Theorem \ref{Thm_MATRest}.
    Of course, $\Ac$ is still almost accurate since it is supersolvable.
    
    Further, from $\Gg$ we can construct a new graph, illustrating the behavior of accuracy with respect to
    restrictions and localizations, as follows. 
    Extending $\Gg$ by one additional vertex $v$ and the
    edges indicated by the dashed lines in Figure \ref{Fig_GraphNotAcc} yields a new chordal graph $\Gg'$.
    The corresponding graphic arrangement $\Bc := \Ac(\Gg')$ has exponents
    $\exp(\Bc) = (0,1,2,\ldots,2,3,3,3) \in \ZZ^{11}$.
    It is not hard to see, by contracting the appropriate edges, 
    % $(1,2),(2,11),(1,3),\ldots$ in order give exponents
    % $(0,1,2,\ldots,2,3,3),(0,1,2,\ldots,2,3),(0,1,2,\ldots,2),\ldots$ 
    that $\Bc$ is accurate. But for $H' = \ker(x_1-x_v) \in \Bc$ (the hyperplane corresponding to the new edge $(1,v)$ of $\Gg'$), 
    we have $\Bc^{H'} = \Ac$. Moreover,  for $$X = \bigcap_{1\leq i < j \leq 10} \ker(x_i-x_j)$$ we also obtain $\Bc_X \cong \Ac$.
    This shows that in general, accuracy is neither inherited by restrictions (even if they are supersolvable),
    nor by localizations.
    
    Lastly, one can show that every free graphic arrangement with fewer than $18$ edges is still accurate. Thus our chordal arrangement 
    $\Ac$ is the smallest one that fails to be accurate. 
\end{example}

%%%%%%%%%%%%%%%%%%%%%%%%%%%%%%%%%%%%%%%%%%%%%%%%%%%%%%%%%%%%%%%%%%%%%%%%%%%%%%%%%%%%%%%

\subsection{Complex reflection arrangements and their restrictions}
\label{ssect:refl}

Let $G \subseteq \GL(V)$ be a finite, 
complex reflection group acting on the complex vector space $V=\CC^\ell$.
The \emph{reflection arrangement} of $G$ in $V$ is the 
hyperplane arrangement $\Ac(G)$ 
consisting of the reflecting hyperplanes 
of the elements in $G$ acting as reflections on $V$.

Terao \cite{Terao1980_FreeUniRefArr} has shown that every 
reflection arrangement $\Ac(G)$ is free 
and that the exponents of $\Ac(G)$
coincide with the coexponents of $G$, 
cf.~\cite[Prop.~6.59 and Thm.~6.60]{OrTer92_Arr}. 

In view of Theorem \ref{thm:otsconj}, it is natural to examine accuracy 
for the larger class of complex reflection arrangements. 
Thanks to Remark \ref{rem_MATRest}(v) and \cite[Prop.~2.12]{roehrle:divfree}, 
accuracy and divisional freeness are compatible with products.

\begin{theorem}
	\label{thm:complex}
	Let $G$ be a complex reflection group with reflection arrangement $\Ac = \Ac(G)$. Then $\Ac$ is accurate 
    if and only if it is divisionally free. This is the case if and only if $G$ has no irreducible factor isomorphic 
    to one of the monomial groups $G(r,r,\ell)$, $r>2$, $\ell>2$, or   
	$ G_{24}, G_{27}, G_{29}, G_{33}, G_{34}$.
\end{theorem}

\begin{proof}
	The result follows readily from Definition \ref{Def_TF},
    the classification of the divisionally free reflection arrangements from \cite[Cor.~4.7]{abe:divfree}, and the exponents of the complex reflection arrangements and their restrictions, e.g.~see \cite[\S 6.4, App.~C]{OrTer92_Arr}. 
\end{proof}

In view of Theorem \ref{thm:complex} and the fact that all restrictions of complex reflection arrangements are free (thanks to \cite[\S 6.4, App.~D]{OrTer92_Arr}, 
\cite{OrTer92_CoxeterArrRestr}, and \cite{hogeroehrle:free}), 
it is natural to investigate accuracy 
among restrictions of complex reflection arrangements, 
not all of which are reflection arrangements again. 

In order to derive our results, we require some further notation.
Orlik and Solomon defined \emph{intermediate 
	arrangements} $\Ac^k_\ell(r)$ in 
\cite[\S 2]{OrSol83_CoxeterArr}
(cf.\ \cite[\S 6.4]{OrTer92_Arr}) which
interpolate between the
reflection arrangements of the monomial groups $G(r,r,\ell)$ and $G(r,1,\ell)$. 
They show up as restrictions of the reflection arrangement
of $G(r,r,\ell')$, for some $\ell'$,  
\cite[Prop.\ 2.14]{OrSol83_CoxeterArr} 
(cf.~\cite[Prop.\ 6.84]{OrTer92_Arr}).

For 
$\ell, r \geq 2$ and $0 \leq k \leq \ell$ the defining polynomial of
$\Ac^k_\ell(r)$ is given by
\[
Q(\Ac^k_\ell(r)) = x_1 \cdots x_k\prod\limits_{\substack{1 \leq i < j \leq \ell\\ 0 \leq n < r}}(x_i - \zeta^nx_j),
\]
where $\zeta$ is a primitive $r$-th root of unity,
so that 
$\Ac^\ell_\ell(r) = \Ac(G(r,1,\ell))$ and 
$\Ac^0_\ell(r) = \Ac(G(r,r,\ell))$. 
For $k \neq 0, \ell$, these are not reflection arrangements
themselves. 

Next we recall
\cite[Props.\ 2.11,  2.13]{OrSol83_CoxeterArr}
(cf.~\cite[Props.~6.82,  6.85]{OrTer92_Arr}).

\begin{proposition}
	\label{prop:intermediate}
	Let $\Ac = \Ac^k_\ell(r)$ for $\ell, r \geq 2$ and $0 \leq k \leq \ell$. Then
	\begin{enumerate}
		\item[(i)] $\Ac$ is free with $$\exp\Ac = (1, r + 1, \ldots, (\ell - 2)r + 1, (\ell - 1)r - \ell + k + 1),$$
		\item[(ii)] for $H \in \Ac$, the type of $\Ac^H$ is given in Table \ref{table2}.
	\end{enumerate}
\end{proposition}

\begin{table}[ht!b]
	\renewcommand{\arraystretch}{1.5}
	\begin{tabular}{llll}\hline
		$k$ & \multicolumn{2}{l}{$\alpha_H$} & Type of $\Ac^H$\\ \hline
		$0$ & arbitrary & & $\Ac^1_{\ell - 1}(r)$\\
		$1, \ldots, \ell - 1$ & $x_i - \zeta x_j$ & $1 \leq i < j \leq k < \ell$ & $\Ac^{k - 1}_{\ell - 1}(r)$\\
		$1, \ldots, \ell - 1$ & $x_i - \zeta x_j$ & $1 \leq i \leq k < j \leq \ell$ & $\Ac^k_{\ell - 1}(r)$\\
		$1, \ldots, \ell - 1$ & $x_i - \zeta x_j$ & $1 \leq k < i < j \leq \ell$ & $\Ac^{k + 1}_{\ell - 1}(r)$\\
		$1, \ldots, \ell - 1$ & $x_i$ & $1 \leq i \leq \ell$ & $\Ac^{\ell - 1}_{\ell - 1}(r)$\\
		$\ell$ & arbitrary & & $\Ac^{\ell - 1}_{\ell - 1}(r)$\\ \hline
	\end{tabular}
	\bigskip
	\caption{Restriction types of $\Ac^k_\ell(r)$}
	\label{table2}
\end{table}

\begin{lemma}
	\label{lem:akl}
	Let $\Ac = \Ac^k_\ell(r)$ for $\ell, r \geq 2$ and $1 \leq k \leq \ell-1$.
	Then 
	\begin{enumerate}
		\item[(i)] for $r =2$, $\Ac$ is accurate;
		\item[(ii)] for $r > 2$, $\Ac$ is accurate if and only if $r + k \ge \ell$.
	\end{enumerate}
\end{lemma}

\begin{proof}
	Thanks to Proposition \ref{prop:intermediate}(ii), the restriction of an intermediate arrangement is again of this kind. So we only need to consider the restriction $\Ac^H$ to a hyperplane. It follows from   Proposition \ref{prop:intermediate} that 
	firstly, if $r + k \ge \ell$, then for $H = \ker  x_i$, the exponents of $\Ac^H$ are all but the largest one of $\Ac$. Secondly, if $r > 2$ and $r + k < \ell$, then there is no restriction 
	$\Ac^H$ with this property. So $\Ac$ is not accurate in this instance.
	
	Finally, for $r = 2$ and $1 \le k < \ell -2$, if we take $H$ to be as in the fourth row of  Table \ref{table2}, then once again the exponents of $\Ac^H$ are all but the largest one of $\Ac$.
\end{proof}

Thanks to  
Lemma \ref{lem:akl}, the following is the smallest example of a non-accurate 
member among the intermediate arrangements $\Ac^k_\ell(r)$.

\begin{example}
	\label{ex:a15}
	By Proposition \ref{prop:intermediate}, 
	$\exp (\Ac^1_5(3)) =	(1,4,7,9,10)$, and the possible  exponents of restrictions are
	$(1,4,7,7)$, $(1,4,7,8)$, and $(1,4,7,10)$. 
	In particular, there is no hyperplane in $\Ac^1_5(3)$ whose 
	restriction has got exponents $(1,4,7,9)$. 
	Consequently, $\Ac^1_5(3)$ is not accurate.
\end{example}

Again thanks to Remark \ref{rem_MATRest}(v) and \cite[Prop.~2.12]{roehrle:divfree}, 
accuracy and divisional freeness are compatible with products. So we can reduce to the case of irreducible $G$ when considering restrictions of reflection arrangements.

\begin{theorem}
	\label{thm:restr}
	Let $G$ be an irreducible complex reflection group with reflection arrangement $\Ac(G)$. Let $\Ac = \Ac(G)^Y$, for $Y \in L(\Ac)\setminus \{V\}$. Then $\Ac$ is accurate
	if and only if one of the following holds:
	\begin{enumerate}[(i)]
		\item $G \ne G(r,r,\ell)$; 
		\item $G = G(r,r,\ell)$ and either $r = 2$ or else $\Ac = \Ac_\ell^k(r)$ with $r + k \ge \ell$ for $r > 2$.
	\end{enumerate}
\end{theorem}

\begin{proof}
	This follows readily from Definition \ref{Def_TF}, Lemma \ref{lem:akl}, and the exponents of the restrictions of the irreducible reflection arrangements, e.g.~see \cite[\S 6.4, App.~C]{OrTer92_Arr}.
\end{proof}

\begin{remark}
	It follows from Theorem \ref{thm:restr} that 
in contrast to the case of complex reflection arrangements from Theorem \ref{thm:complex}, there are 
divisionally free restrictions of the latter that are not accurate. For, thanks to 
\cite[Thm.~5.6]{abe:divfree}, all $\Ac^k_\ell(r)$ are divisionally free 
for $\ell, r \geq 3$ and $k \geq 1$.

Nevertheless, from the classification of the divisionally free restrictions of reflection arrangements from \cite[Thm.~3.3]{roehrle:divfree} and the exponents of the restrictions of the reflection arrangements, e.g.~see \cite[\S 6.4, App.~C]{OrTer92_Arr}, it follows from Theorem \ref{thm:restr} that if $G \ne G(r,r,\ell)$, then $\Ac(G)^Y$ is accurate
if and only if it is divisionally free.
\end{remark}

Note that $G(2,2,\ell)$ is the Coxeter group of type $D_\ell$ and thus the arrangements
$\Ac^k_\ell(2)$ are restrictions of Coxeter arrangements  of type $D$.
As a consequence of Lemma \ref{lem:akl}, all such are accurate.

The following consequence of Theorem \ref{thm:restr} shows that accuracy for Coxeter arrangements (Theorem \ref{thm:otsconj})
 extends to their restrictions.

\begin{corollary}
	\label{cor:restr-Coxeter}
	Restrictions of Coxeter arrangements are accurate. 
\end{corollary}

Observe that accuracy  of Coxeter arrangements, Theorem \ref{thm:otsconj}, is a consequence of Theorem \ref{Thm_MATRest} and the fact that all such are MAT-free, \cite{CunMue19_MATfree}. In contrast, restrictions of Coxeter arrangements are not MAT-free in general, see \cite[Ex.~22]{CunMue19_MATfree}.
So Corollary \ref{cor:restr-Coxeter} does not follow from Theorem \ref{Thm_MATRest} and is independent from Theorem \ref{thm:otsconj}.

Finally, we present an example among the intermediate arrangements which 
shows that accuracy is not compatible with taking localizations.

\begin{example}
	\label{ex:local}
	Let $\Ac = \Ac^1_\ell(r)$ for $\ell \ge 4$ and $r \ge \ell -1$.
	Let 
	\[
	X = \bigcap\limits_{\substack{2 \leq i < j \leq \ell\\ 0 \leq n < r}} \ker(x_i - \zeta^nx_j),
	\]
	where $\zeta$ is a primitive $r$-th root of unity.
	Then $\Ac_X \cong \Ac^0_{\ell-1}(r)$.
	By Theorems \ref{thm:complex} and \ref{thm:restr},
	$\Ac$ is accurate but $\Ac_X  \cong \Ac(G(r,r,\ell-1))$ is not.
\end{example}

%%%%%%%%%%%%%%%%%%%%%%%%%%%%%%%%%%%%%%%%%%%%%%%%%%%%%%%%%%%%%%%%%%%%%%%%%%%%%%%%%%%%%%%
%%%%%%%%%%%%%%%%%%%%%%%%%%%%%%%%%%%%%%%%%%%%%%%%%%%%%%%%%%%%%%%%%%%%%%%%%%%%%%%%%%%%%%%

\newcommand{\etalchar}[1]{$^{#1}$}
\providecommand{\bysame}{\leavevmode\hbox to3em{\hrulefill}\thinspace}
\providecommand{\MR}{\relax\ifhmode\unskip\space\fi MR }
% \MRhref is called by the amsart/book/proc definition of \MR.
\providecommand{\MRhref}[2]{%
  \href{http://www.ams.org/mathscinet-getitem?mr=#1}{#2}
}
\providecommand{\href}[2]{#2}

\end{document}